\documentclass[12pt,reqno]{article}
\usepackage{a4wide}

\usepackage{amssymb, amsmath, hyperref}

\usepackage{pdfsync,cancel}

\usepackage[normalem]{ulem}

\usepackage{relsize}

\usepackage{bbm}

 \hsize \textwidth
 \advance \hsize by -\marginparwidth
 \evensidemargin \oddsidemargin
\usepackage{amsmath}    % remove % for AMS-LaTeX
\usepackage{amssymb}

\usepackage{graphics,graphicx,fancyhdr,color}

\providecommand{\keywords}[1]{\textbf{\textit{Keywords:}} #1}

\newtheorem{df}{Definition}[section]
\newtheorem{lm}[df]{Lemma}%[section]
\newtheorem{lemma}[df]{Lemma}
\newtheorem{prop}[df]{Proposition}%[section]
\newtheorem{thm}[df]{Theorem}%[section]
\newtheorem{cor}[df]{Corollary}%[section]

\newcommand{\Rm}{{\mathbb R}}
\newcommand{\Tm}{{\mathbb T}}

\makeatletter \@addtoreset{equation}{section}

\newcommand{\bes}{\begin{displaymath}}
\newcommand{\ees}{\end{displaymath}}
\newcommand{\be}{\begin{equation}}
\newcommand{\ee}{\end{equation}}
\newcommand{\ba}{\begin{eqnarray}}
\newcommand{\ea}{\end{eqnarray}}
\newcommand{\bas}{\begin{eqnarray*}}
\newcommand{\eas}{\end{eqnarray*}}
\newcommand{\@Bbb}[1]{\ensuremath{\mathbb #1}}

\newcommand{\B}{{\@Bbb B}}
\newcommand{\C}{{\@Bbb C}}
\newcommand{\E}{{\mathbb E}}
\newcommand{\F}{{\@Bbb F}}
\renewcommand{\P}{{\mathbb P}}
\newcommand{\bbP}{{\P}}
\newcommand{\bbE}{{\mathbb E}}
\newcommand{\Q}{{\@Bbb Q}}
\newcommand{\bQ}{{\@Bbb Q}}
\newcommand{\N}{{\@Bbb N}}
\newcommand{\fT}{{\frak T}}

\newcommand{\bbR}{{\@Bbb R}}
\newcommand{\W}{{\@Bbb W}}

\newcommand{\bbZ}{{\@Bbb Z}}
\newcommand{\bbT}{{\@Bbb T}}

\newcommand{\la}{\lambda}
\newcommand{\al}{\alpha}

\newcommand{\si}{\sigma}

\newcommand{\Om}{\Omega}
\newcommand{\om}{\omega}

\newcommand{\eps}{\varepsilon}

\newcommand{\@s}[1]{\ensuremath{\mathcal #1}}
\newcommand{\cA}{\@s A}
\newcommand{\cB}{\@s B}
\newcommand{\cC}{\@s C}
\newcommand{\cD}{\@s D}
\newcommand{\cE}{\@s E}
\newcommand{\cF}{\@s F}
\newcommand{\cG}{\@s G}
\newcommand{\cH}{\@s H}
\newcommand{\cI}{\@s I}
\newcommand{\cJ}{\@s J}

\newcommand{\cK}{\@s K}
\newcommand{\cL}{\@s L}
\newcommand{\cN}{\@s N}
\newcommand{\cM}{\@s M}
\newcommand{\cO}{\@s O}
\newcommand{\cP}{\@s P}
\newcommand{\cR}{\@s R}
\newcommand{\cS}{\@s S}
\newcommand{\cT}{\@s T}
\newcommand{\cV}{\@s V}
\newcommand{\cW}{\@s W}
\newcommand{\cX}{\@s X}
\newcommand{\cY}{\@s Y}
\newcommand{\cZ}{\@s Z}

\newcommand{\@bm}[1]{\ensuremath{\mathbf #1}}
\newcommand{\bma}{\@bm a}
\newcommand{\bmb}{\@bm b}
\newcommand{\bmc}{\@bm c}
\newcommand{\bmd}{\@bm d}

\newcommand{\bme}{\@bm e}
\newcommand{\bmf}{\@bm f}
\newcommand{\bmg}{\@bm g}
\newcommand{\bmh}{\@bm h}
\newcommand{\bmi}{\@bm i}
\newcommand{\bmj}{\@bm j}
\newcommand{\bmk}{\@bm k}
\newcommand{\bml}{\@bm l}
\newcommand{\bmm}{\@bm m}
\newcommand{\bmn}{\@bm n}
\newcommand{\bmo}{\@bm o}
\newcommand{\bmp}{\@bm p}
\newcommand{\bmq}{\@bm q}
\newcommand{\bmr}{\@bm r}
\newcommand{\bms}{\@bm s}
\newcommand{\bmt}{\@bm t}
\newcommand{\bmu}{\@bm u}
\newcommand{\bmw}{\@bm w}
\newcommand{\bmv}{\@bm v}
\newcommand{\bmx}{\@bm x}
\newcommand{\bx}{\@bm x}

\newcommand{\bmy}{\@bm y}
\newcommand{\bz}{\@bm z}
\newcommand{\bbN}{\mathbb N}

\newcommand{\by}{\@bm y}

\newcommand{\bmzero}{\@bm 0}

\newcommand{\ga}{\gamma}

\newcommand{\@g}[1]{\ensuremath{\mathfrak #1}}
\newcommand{\gA}{\@g A}
\newcommand{\gD}{\@g D}
\newcommand{\gJ}{\@g J}
\newcommand{\gF}{\@g F}
\newcommand{\gM}{\@g M}
\newcommand{\gR}{\@g R}

\newcommand{\commentout}[1]{{}}

\newcommand{\proof}{{\bf Proof. }}
\newcommand{\qed}{$\Box$}

\makeatother
\begin{document}

\title
  {Fractional diffusion limit for a kinetic equation with
  an interface\thanks{{\bf Acknowledgement.} TK  acknowledges the support of the National Science Centre under the 
    NCN grant 2016/23/B/ST1/00492. LR  was partially supported by the NSF grant DMS-1613603
    and ONR grant N00014-17-1-2145.
SO was partially supported by ANR-15-CE40-0020-01 grant LSD.
TK expresses his gratitude to Prof. K.~Bogdan for
elucidating conversations on the subject of the paper and to the
Department of Mathematics of the Technical University of Wroc\l aw for its
hospitability.}} % -  probabilistic approach}
%\titlerunning{Superdiffusive Limit with thermostatted interface}

\author{Tomasz Komorowski\thanks{Institute of Mathematics, Polish Academy Of Sciences,
ul. \'{S}niadeckich 8,   00-956 Warsaw, Poland, 
e-mail: {\tt komorow@hektor.umcs.lublin.pl}}
\and
Stefano Olla\thanks{CEREMADE, UMR-CNRS, Universit\'e de Paris Dauphine, PSL Research University
{ Place du Mar\'echal De Lattre De Tassigny, 75016 Paris, France},
e-mail:{ \tt olla@ceremade.dauphine.fr}}
\and
 Lenya Ryzhik\thanks{{ Mathematics Department, Stanford University, Stanford, CA 94305, USA} ,
{email: {\tt ryzhik@stanford.edu}}}}

\maketitle

\begin{abstract}

We consider the limit of  a  linear kinetic equation, with reflection-transmission-absorption at an
interface, with a degenerate scattering kernel. The equation arise from a microscopic chain of oscillators in contact
with a heat bath. In the absence of the interface, the solutions exhibit 
a superdiffusive behavior in the long time limit.  With the interface,  the long time limit is the unique solution of a 
version of the fractional in space heat equation,
with reflection-transmission-absorption at the interface.  
The limit problem corresponds to a certain stable process that is
either absorbed, reflected, or transmitted upon crossing the interface.
\end{abstract}

%\date{\today {\bf File: {\jobname}.tex.}} 

\section{Introduction}

\label{intro}

We consider a linear phonon Boltzmann equation with an
interface. This
equation describes the evolution of the energy density $W(t,y,k)$ of   phonons   at time $t\ge0$,
spatial position $y\in\bbR$  and the frequency 
$k\in\bbT=[-1/2,1/2]$  with identified endpoints.
Outside the interface,  located at $y=0$, the density satisfies the kinetic equation
\begin{equation}
  \label{eq:8}
\begin{aligned} 
 &\partial_tW(t,y,k) + \bar\om'(k) \partial_y W(t,y,k) = \ga_0 L_k W(t,y,k), 
 % \quad (t,y,k)\in\bbR_+\times \bbR_*\times\bbT_*,
 \\
&
W(0,y,k)=W_0(y,k).
\end{aligned}
\end{equation}
% We use the notation
% $
% \bbR_+=(0,+\infty)$, $\bbR_-=(-\infty,0)$,~$\bbR_*=\bbR\setminus\{0\}$,~$\bar\bbR_\pm=[0,\pm\infty)$
% as well as~$\bbT_*=\bbT\setminus\{0\}$,~and $\bbT_\pm=[k:\,0<\pm k<1/2].
% $
%Function $W_0(y,k)$ is the density of the initial
%distribution of phonons in the position and frequency domain.
We denote by $\om:\bbT\to[0,+\infty)$ the dispersion relation, and set
{the group velocity of the phonon}
$
\bar\om'(k):={\om'(k)}/{(2\pi)}$, $k\in\bbT$. 
The parameter $\ga_0>0$ represents the phonon scattering rate, and 
the scattering operator $L_k$, acting only on the $k$-variable,
is given~by 
\begin{equation}
\label{L}
L_kF(k):= {\displaystyle \int_{\bbT}}R(k,k')
\left[F\left(k'\right) - F\left(k\right)\right]dk',\quad k\in\bbT,
\end{equation}
{for a bounded and measurable  function $F$}.
%for $F$ belongs to $ B_b(\bbT)$ - the set  of bounded measurable,  real valued functions.
%Here
%$$
%\bar\om'(k)=\frac{\om'(k)}{2\pi},\quad k\in\bbT,
%$$
%where  {\em the dispersion relation} $ \om:\bbT\to[0,+\infty)$.
%
%

When there is no interface, this is the Kolmogorov equation for a classical jump process.
% It also appears as the macroscopic limit of a system of oscillators driven by a random noise
% that conserves {energy, momentum and volume} \cite{bos}. This microscopic problem has been recently
% considered in ~\cite{kors} with a thermostat at a fixed temperature $T\ge 0$ at the location~$y=0$, 
% so that the phonons may be 
% emitted, reflected or transmitted by thermostat, and the corresponding macroscopic interface conditions have been
% obtained, in the absence of the bulk scattering, corresponding to $\gamma_0=0$ in (\ref{eq:8}). 
The interface conditions  prescribe the outgoing phonon density in terms of what comes to the interface: 
\begin{equation}\label{feb1408}
W(t,0^+, k)=p_-(k)W(t,0^+, -k)+p_+(k)W(t,0^-,k)+T\mathfrak g(k), \hbox{ for $0< k\le 1/2$},
\end{equation}
and
\begin{equation}\label{feb1410}
W(t,0^-, k)=p_-(k)W(t,0^-,-k) + p_+(k)W(t,0^+, k)+T\mathfrak g(k), \quad \hbox{ for $-1/2< k< 0$},
\end{equation}
% - in agreement with the convention \eqref{conv} and condition \eqref{tot} we also let
% \begin{equation}
% \label{023101-19}
% W(t,0^\pm, 0)\equiv T,\quad t\ge0.
% \end{equation}
with the energy balance
\begin{equation}
\label{012304}
p_+(k)+p_-(k)+  \mathfrak g(k)=1. 
\end{equation}
Here, $p_-(k)$, $p_+(k)$ and $\mathfrak g(k)$ are, respectively,
% \sout{the reflection and
% transmission coefficients across the interface,  corresponding to} the
the probabilities of the phonon being reflected, transmitted or absorbed, while
$T\mathfrak g(k)$ is the phonon production rate at the interface. 
% \sout{One can also interpret  $0\le \mathfrak g(k)\le 1$ as the absorption rate of the
% frequency $k$ phonon by the interface}

We  assume  
that the absorption coefficient $ \frak g(k)$ and the reflection-transmission
coefficients~$p_\pm(k)$ are positive, continuous, 
even functions, satisfying \eqref{012304} and  such that
\begin{equation}
\label{non-deg2a}
\lim_{k\to 0+}\frak g(k)= {\frak g_0>0}, \quad \lim_{k\to 0+}p_\pm(k)=p_\pm,
\end{equation}
and
there exist $C_0,\ga>0$ such that
\begin{equation}
\label{non-deg1}
|p_\pm(k)-p_\pm|\le  C_0| k|^{\ga},\quad k\in\bbT.
\end{equation} 
%This equation has been obtained in \cite{bos},
%with no thermostat present, as the Boltzmann-Grad limit of the energy density 
%function for a microscopic model of a heat conductor consisting  of a one dimensional chain of harmonic oscillators, with
%inter-particle scattering conserving the energy and  volume.

The large scale limit of the kinetic equation without an interface has been considered 
in~\cite{bb,jko,mmm}. The corresponding rescaled problem, with $N\to+\infty$,
is
\begin{equation}
\label{kinetic-sc0}
\begin{aligned}
&\dfrac{1}{N^{\al}}\partial_t
  W_N(t,y,k)+\dfrac{1}{N}\bar\om'(k)\partial_yW_N(t,y,k)=\ga_0
  L_kW_N(t,y,k), %&(t,y,k)\in\bbR_+\times \bbR_*\times \bbT_*,\\
\\
&W_N(0,y,k)=W_0(y,k),
\end{aligned}
\end{equation}
with  an appropriate exponent $\al>0$,
with $\alpha=2$ corresponding to the classical diffusive scaling. 
An important feature of the phonon scattering is that 
the total scattering kernel
\begin{equation}
\label{tot0}
R(k):=\int_{\bbT}R(k,k')dk'
\end{equation}
degenerates at $k=0$ % so that $R(k)\sim |k|^2$,
-- phonons at a low frequency scatter much less. 
The correct choice of the time rescaling exponent $\al$ depends then
on the properties of the dispersion relation. In the optical case, when 
${\bar\om'(k)}\sim k$, $|k|\ll1$, so that the low frequency phonons not only scatter less but also travel slower, 
the scaling in \eqref{kinetic-sc0}
is diffusive, so that~$\al=2$ and $W_N(t,y,k)$ converges as $N\to+\infty$ to
the solution to a heat equation 
\begin{equation}
\label{heat0}
%\left\{
%\begin{array}{ll}
\partial_t W(t,y)=\hat c\partial_{yy}^2 W(t,y), %&(t,y)\in\bbR_+\times \bbR_*,\\
\end{equation}
with the initial condition  
\[
W(0,y)=\mathlarger{\int}_{\bbT}W_0(y,k)dk,
\]
and an appropriate diffusion coefficient  $\hat c>0$.

When, on the other hand, the dispersion relation is acoustic, so that
${\bar\om'(k)}\sim {\rm sign}\,k$, for~$|k|\ll1$, and the phonons at low frequency scatter less but move as fast
as other phonons, 
then the scaling is
super-diffusive, with $\al=3/2$ and the limit of $W_N(t,y,k)$ as $N\to+\infty$ satisfies the fractional
heat equation
\begin{equation}
\label{heat1}
\begin{aligned}
&\partial_t W(t,y)=-\hat c|\partial_{yy}^2|^{3/4} W(t,y), \\ %&(t,y)\in\bbR_+\times \bbR_*,\\
&
W(0,y)=\mathlarger{\int}_{\bbT}W_0(y,k)dk,
\end{aligned}
\end{equation}
with an appropriate fractional diffusion coefficient
$\hat c>0$. In both cases the limit $W(t,y)$ does not depend on the frequency $k$.
Results of this type under various assumptions on
the scattering kernel (but without an interface present) have been proved in~\cite{bb,AMP,FS,jko,mmm}.

Our interest here is to understand the long time behavior of the solutions to the kinetic equation
with an acoustic
dispersion relation  in the presence of
the interface, so that (\ref{kinetic-sc0}) holds away from $y=0$, and the 
interface conditions (\ref{feb1408})-(\ref{feb1410}) for $W_N$ hold at $y=0$. 
 We allow the total scattering rate to degenerate as
$R(k)\sim|\sin(\pi k)|^{\beta}$ for some $\beta>1$. The
case~$\beta\in(0,1)$ has been considered in~\cite{kob}, 
with the initial condition that is a local perturbation of the the equilibrium solution
$W(t,y,k)\equiv T$. It 
leads to a diffusive scaling and the limit described by a heat
equation \eqref{heat1}, with  a pure absorption interface
condition  $W(t,0)=T$.   In that situation, the degeneracy of scattering 
at low frequencies is
not strong enough to prevent the diffusive behavior. 

In order to formulate our main result, let us make some assumptions
on the scattering kernel, reflection-transmission-absorption coefficients and the initial condition.
We assume that the scattering kernel is symmetric
\begin{equation}
\label{sym}
R(k,k')=R(k',k),
\end{equation}
positive, except possibly at $k=0$:
\begin{equation}
\label{pos}
R(k,k')>0,\quad k,k'\neq 0,
\end{equation}
and the total scattering kernel has the asymptotics
\begin{equation}
\label{tot}
R(k):=\int_{\bbT}R(k,k')dk'\sim R_0|\sin(\pi k)|^{\beta},\quad |k|\ll 1,
\end{equation}
with some $\beta\ge 0$ and $R_0>0$. We  also assume that the normalized cross-section
\begin{equation}
\label{ext}
p(k,k'):=\frac{R(k,k')}{R(k)},\quad k,k'\neq 0,
\end{equation}
extends to a $C^\infty$ function on $\bbT^2$. Note that
\begin{equation}\label{apr202}
\int_{\Tm} p(k,k')dk'=1,\hbox{ for all $k\neq 0$.}
\end{equation}
 For the dispersion relation, we assume that it is acoustic, that is, 
\begin{equation}
\label{acoust}
\om(k)\sim 2\om_0'|\sin(\pi k)|,\quad |k|\ll 1,
\end{equation}
with some $\om_0'>0$, and that $\om(k)$ is even in $k$.

To make the precise assumptions on $W_0(y,k)$, we will use the notation
$
\bbR_+=(0,+\infty)$, $\bbR_-=(-\infty,0)$,~$\bbR_*=\bbR\setminus\{0\}$,~$\bar\bbR_+=[0,+\infty)$, ~$\bar\bbR_-=(-\infty,0]$,
as well as~$\bbT_*=\bbT\setminus\{0\}$,~and $\bbT_\pm=[k:\,0<\pm k<1/2].
$
Given $T$, we let ${\cal C}_T$ be a
subclass of $C_b(\bbR_*\times\bbT_*)$  of functions~$F$ that can be
continuously extended to $\bar\bbR_\pm\times\bbT_*$ and satisfy the
interface conditions
\begin{align}
\label{feb1408aa}
&
F(0^+,k)=p_-(k)F(0^+, -k)+p_+(k)F(0^-,k)+{\frak g}(k)T,& \hbox{ for $0< k\le 1/2$},\\
&
F(0^-, k)=p_-(k)F(0^-,-k) + p_+(k)F(0^+, k) +{\frak g}(k)T,& \hbox{ for $-1/2< k< 0$.}\nonumber
\end{align}
 %\begin{remark}
%\label{rm012603-19}
{Note that  $F\in{\cal C}_T$ if and only if $F-T'\in{\cal C}_{T-T'}$
for some $T'$, because of (\ref{012304}).} 
%\end{remark}

In the presence of the interface, the fractional diffusion equation (\ref{heat1}) is replaced by the following non-local
equation
\begin{equation}\label{apr204}
\begin{aligned}
&\partial_tW(t,y)=\hat c\int_{yy'>0}q_\beta(y-y')[W(t,y')-W(t,y)]dy'+\hat c{\frak g}_0\int_{yy'<0}q_\beta(y-y')[T-W(t,y)]dy'\\
&+\hat cp_-\int_{yy'<0}q_\beta(y-y')[W(t,-y')-W(t,y)]dy'+
\hat cp_+\int_{yy'<0}q_\beta(y-y')[W(t,y')-W(t,y)]dy'.
\end{aligned}
\end{equation}
Here, $\hat c$ is a fractional diffusion coefficient given by (\ref{hatc}) below, $p_\pm$ and $\frak g_0$ are as in (\ref{non-deg2a}), and
\begin{equation}
\label{cal}
q_\beta(y)=\frac{c_\beta}{|y|^{2+1/\beta}},~~~ 
c_\beta:=\frac{2^{1+1/\beta}\Gamma\left(1+1/(2\beta)\right)}{\sqrt{\pi}\Gamma\left(-1/2+1/(2\beta)\right)}.%\quad y\in\bbR_*,
\end{equation}
%and the normalization constant
%\begin{equation}
%\label{cal}
% c_\beta:=\frac{2^{1+1/\beta}\Gamma\left(1+1/(2\beta)\right)}{\sqrt{\pi}\Gamma\left(-1/2+1/(2\beta)\right)}.
%\end{equation}
As we explain below, equation (\ref{apr204}) automatically incorporates the interface conditions. 
%The precise meaning of a weak solution to (\ref{apr204}) 
%is specified in Definition $\ref{df011803-19}$ below. 
% Informally, it says that $W(t,y)$ is a weak solution of (\ref{apr204}) that satisfies
% the correct interface conditions and lies in a certain energy space. 
Our main result is as follows. 
\begin{thm}
\label{main-thm}
In addition to the above assumptions about the scattering kernel
$R(\cdot,\cdot)$ and the  dispersion relation $\om(\cdot)$,
suppose that $\beta>1$ and $W_0\in
{\cal C}_T$, and let $W_N(t,y,k)$ be the solution to (\ref{kinetic-sc0}) with $\alpha=1+1/\beta$. Then, we have
\begin{equation}
\label{conv}
\lim_{N\to+\infty}\int_{\bbR\times \bbT}W_N(t,y,k)G(y,k)dy
dk=\int_{\bbR\times \bbT}W(t,y)G(y,k)dy dk,
\end{equation}
for any
$t>0$, and any test function $G\in C^\infty_0(\bbR\times\bbT)$.  The limit $W(t,y)$
is a weak solution of equation \eqref{apr204}, in the sense of Definition $\ref{df011803-19}$,
with the initial condition
\begin{equation}
\label{022603-19}
\bar W_0(y):=\int_{\bbT}W_0(y,k)dk
\end{equation}
and the fractional diffusion coefficient
\begin{equation}
\label{hatc}
\hat c:=\frac{2\pi^{\beta}(\om'_0)^{1+1/\beta}}{\beta (\ga_0R_0)^{1/\beta}}\,{\rm p.v.}\int_{\bbR}\frac{(e^{i\la}-1)d\la}{|\la|^{2+1/\beta}}.
\end{equation}
\end{thm}

% If, in addition, we
% suppose that $W_0\in
% {\cal C}_T\cap L^1(\bbR\times \bbT)$, then  $W(t,y)$ satisfies
% \eqref{apr204}, in the sense of Definition $\ref{df011803-19}$, with $f\equiv 0$ and the initial data $\bar W_0$.

The proof of this theorem proceeds as follows: as we have mentioned,  
the kinetic equation~(\ref{kinetic-sc0}) is the Kolmogorov
equation for a Markov process $(Z_N(t),K_N(t))$, where the frequency $K_N(t)$ is a 
certain jump process and the spatial
component $Z_N(t)$ is the time integral of $\bar\omega'(K_N(t))$. 
This process can be generalized to incorporate the 
reflection-transmission-absorption at the interface.
Similarly, we show that (\ref{apr204}) is a Kolmogorov equation for a certain
stable  process $\zeta(t)$ that undergoes reflection-transmission-absorption at
the interface. 
{We prove that $Z_N(t,y)$ converges to $\zeta(t)$ in law. This shows that $W_N(t,y,k)$
converges to a weak solution $W(t,y)$ of (\ref{apr204}), such
that~$W(t,y)=\E[\bar W_0(\zeta(t))|\zeta(0)=y]$.

Theorem \ref{main-thm} identifies the limit as a weak solution only in the sense of Definition \ref{df011803-19} below, that
does not characterize its behaviour at the interface. In order to obtain this information we need to prove that the limit
belongs to a class of functions that satisfy a certain 
regularity condition at the interface (see \eqref{Hnorm} and \eqref{eq:1}).
When it is imposed the solution is unique.

In Theorem \ref{thm012205-19} we prove that the weak
solution obtained in Theorem \ref{main-thm}
belongs to this regularity class, under the 
further assumption that the initial condition $\bar W_0(y)$
belongs to $L^1$ add to an additive constant.
{% It can  be seen, see Section \ref{eq:1},
  % that equation  (\ref{apr204}) has a unique solution in the space of
  % functions satisfying condition \eqref{012205-19}.    
 %  In Section
% \ref{sec:proof-theor-refthm01}  we prove that the limit $W(t,y)$ described in Theorem
% \ref{main-thm} satisfies this condition.
To this end, 
we construct another approximation $\zeta_a(t)$ of $\zeta(t)$ that converges in law to $\zeta(t)$ as $a\to 0^+$ and 
\[
W_a(t,y)=\E\left[\bar W_0(\zeta_a(t))|\zeta_a(0)=y\right]\to W(t,y)=\E[\bar W_0(\zeta(t))|\zeta(0)=y].
\]
However, we ensure that $W_a(t,y)$ also satisfies an energy estimate,
see \eqref{022003-19zz1} below, thus so does
$W(t,y)$ in the limit.
%which in turn proves that $W(t,y)$ satisfies  \eqref{012205-19}.
}

A kinetic problem with similar conditions at the interface
appears as the macroscopic limit of a system of oscillators driven by a random noise
that conserves {energy, momentum and volume} \cite{bos}. This microscopic model has been recently
considered in ~\cite{kors}, with a thermostat at a fixed temperature $T\ge 0$ acting on one particle, 
so that the phonons may be 
emitted, reflected or transmitted, and the corresponding macroscopic interface conditions have been
obtained, in the absence of the bulk scattering, corresponding to $\gamma_0=0$ in (\ref{eq:8}).
It is believed that the above macroscopic interface conditions also hold in  the presence of
interior microscopic scattering when $\ga_0>0$. However, for the absorbing probability arising from
this microscopic dynamics, we have $\frak g_0 = 0$ (cf \eqref{non-deg2a}).
This generates a different interface condition for the macroscopic limit \cite{kor20}
from the one obtained here. 

There seem to be few results on a fractional diffusion limit for kinetic equations in the
presence of an interface. In \cite{ce}, the
case of  absorbing, or reflecting boundary, but with the operator~$L$ 
that  itself is a generator of a fractional diffusion, has been
considered. Another situation, closer to ours, is
a subject of \cite{cmp}, where the convergence of solutions to kinetic
equations with the diffusive reflection conditions  on the boundary is investigated. 
This condition   is,
however, different from our interface condition that concerns reflection-transmission-absorption. Also, in contrast to our situation, the results of
\cite{cmp}, %see Theorems 1.2 and 1.3 of \cite{cmp}, 
do not establish
the uniqueness of the limit for solutions of the kinetic equation,
stating only that it satisfies a certain fractional diffusive equation
with a boundary condition. The question of the uniqueness of the
solution for the limiting equation seems to be left open, see the
remark after Theorem 1.2 in~\cite{cmp}. We mention 
here also a result of \cite{bgm}, where solutions of a stationary (time independent)
linear kinetic equation are considered. The spatial domain is a half-space, with the absorbing-reflecting-emitting boundary, of a different type than in the present paper,  and frequencies belong  to a cylindrical domain. It has been shown that under an appropriate scaling the solutions converge to a harmonic function corresponding to a Neumann boundary, fractional Laplacian with exponent $1/2$.

\section{Some preliminaries}
%On the kinetic equation with    reflection-transmission-absorption at an interface }
\label{sec2}

%\subsection{Definition of a solution} 
%of a kinetic equation with reflection-transmission-absorption}
%In what follows we shall also use the notation
%$$
%\bbR_-:=(-\infty,0),\quad \bbT_\pm:=[k:\,0<\pm k<3/4]. 
%$$
%We denote by ${\cal C}$ the class of functions
%$W:\bbR_+\times\bbR_*\times \bbT_*\to\bbR$ that are bounded,
%continuous and the following conditions hold:
%\begin{itemize}
%\item[(1)] 
%The restrictions of $W$ to  
%$\bbR_+\times\bbR_\iota\times \bbT_*$, $\iota\in\{-,+\}$, can be extended
%to bounded and continuous functions on 
%$\bar\bbR_+\times\bar\bbR_{\iota}\times \bar\bbT_{\iota'}$,
%$\iota'\in\{-,+\}$.
%\item[(2)] For each $(t,y,k)\in \bbR_+\times\bbR_*\times  \bbT_*$ fixed, the function
%$W(t+s,y+\bar\om'(k) s,k)$ is of the~$C^1$ class in
%the $s$-variable in a  neighborhood of $s=0$, so that the directional derivative  
%%We denote by 
%\begin{equation}
%\label{Dt}
%D_tW(t,y,k)=\left(\partial_t+\bar\om'(k)\partial_y\right)W(t,y,k):=\frac{d}{ds}_{|s=0} W(t+s,y+\bar\om'(k) s,k)
%\end{equation}
%%the directional derivative of $W $ and assume that it 
%is bounded in $\bbR_+\times\bbR_*\times \bbT_*$.
%\end{itemize}

\subsubsection*{The classical solution of the kinetic interface problem}

We start with the definition of a classical solution to the kinetic interface problem.
\begin{df}
\label{df013001-19}
We say that a function
$W(t,x,k)$, $t\ge 0$, $x\in\Rm$, $k\in\bbT_*$, %\bar \bbR_+\times\bbR\times \bbT_*\to\bbR$ 
is a classical
solution to equation \eqref{eq:8} with the interface conditions
\eqref{feb1408} and \eqref{feb1410}, if it is bounded and continuous on
$\bbR_+\times\bbR_*\times \bbT_*$, and the
following conditions hold:
\begin{itemize}
\item[(1)] 
The restrictions of $W$ to  
$\bbR_+\times\bbR_\iota\times \bbT_*$, $\iota\in\{-,+\}$, can be extended
to bounded and continuous functions on 
$\bar\bbR_+\times\bar\bbR_{\iota}\times \bar\bbT_{\iota'}$,
$\iota'\in\{-,+\}$.
\item[(2)] For each $(t,y,k)\in \bbR_+\times\bbR_*\times  \bbT_*$ fixed, the function
$W(t+s,y+\bar\om'(k) s,k)$ is of the~$C^1$ class in
the $s$-variable in a  neighborhood of $s=0$, and the directional derivative  
%We denote by 
\begin{equation}
\label{Dt}
D_tW(t,y,k)=\left(\partial_t+\bar\om'(k)\partial_y\right)W(t,y,k):=\frac{d}{ds}_{|s=0} W(t+s,y+\bar\om'(k) s,k)
\end{equation}
%the directional derivative of $W $ and assume that it 
is bounded in $\bbR_+\times\bbR_*\times \bbT_*$
and satisfies
\begin{equation}
  \label{eq:8a}
 \begin{array}{ll}
 D_tW(t,y,k) = \ga_0 L_k W(t,y,k), &
\quad (t,y,k)\in\bbR_+\times \bbR_*\times\bbT_*,
\end{array}
\end{equation}
together with \eqref{feb1408} and \eqref{feb1410}  and
\begin{equation}
\label{010102-19}
\lim_{t\to0+}W(t,y,k)=W_0(y,k),\quad (y,k)\in\bbR_*\times\bbT_*.
\end{equation}
\end{itemize}
\end{df}

%We assume that the scattering kernel is symmetric
%\begin{equation}
%\label{sym}
%R(k,k')=R(k',k),
%\end{equation}
%positive, except possibly at $k=0$:
%\begin{equation}
%\label{pos}
%R(k,k')>0,\quad (k,k')\in\bbT_*^2
%\end{equation}
%and the total scattering kernel has the asymptotics
%\begin{equation}
%\label{tot}
%R(k):=\int_{\bbT}R(k,k')dk'\sim R_0|\sin(\pi k)|^{\beta},\quad |k|\ll 1,
%\end{equation}
%with some $\beta\ge 0$ and $R_0>0$. We  also assume that the normalized cross-section
%\begin{equation}
%\label{ext}
%p(k,k'):=\frac{R(k,k')}{R(k)},\quad  (k,k')\in\bbT_*^2,
%\end{equation}
%extends to a $C^\infty$ function on $\bbT^2$. Note that
%\begin{equation}\label{apr202}
%\int_{\Tm_*} p(k,k')dk'=1,\hbox{ for all $k\in\Tm_*$.}
%\end{equation}
% For the dispersion relation, we assume that it is acoustic, that is, 
%\begin{equation}
%\label{acoust}
%\om(k)\sim 2\om_0'|\sin(\pi k)|,\quad |k|\ll 1,
%\end{equation}
%with some $\om_0'>0$, and that $\om(k)$ is even in $k$. 

%\begin{df}
%\label{df012603-19}
%Given $T\ge0$, let ${\cal C}_T$ be a
%subclass of $C_b(\bbR_*\times\bbT_*)$ that consists of functions~$F$ that can be
%continuously extended to $\bar\bbR_\pm\times\bbT_*$ and satisfy the
%interface conditions
%\begin{align}
%\label{feb1408aa}
%&
%F(0^+,k)=p_-(k)F(0^+, -k)+p_+(k)F(0^-,k)+{\frak g}(k)T,& \hbox{ for $0< k\le 1/2$},\\
%&
%F(0^-, k)=p_-(k)F(0^-,-k) + p_+(k)F(0^+, k) +{\frak g}(k)T,& \hbox{ for $-1/2< k< 0$.}\nonumber
%\end{align}
%\end{df}
%%\begin{remark}
%%\label{rm012603-19}
%Note that  $F\in{\cal C}_T$ if and only if $F-T\in{\cal C}_0$, because of (\ref{012304}). 
%%\end{remark}
The following result is standard. 
\begin{prop}
\label{prop013001-19}
Suppose that $W_0\in{\cal C}_T$. Then, under the above
hypotheses on the scattering kernel $R(k,k')$ and the
dispersion relation $\om(k)$, there exists a
unique classical
solution to  equation \eqref{eq:8} with the interface conditions
\eqref{feb1408} and \eqref{feb1410} in the sense of Definition $\ref{df013001-19}$.   
\end{prop}
The existence part is proved in Appendix \ref{app} below, while
uniqueness follows from Proposition~\ref{prop010701-19a}, also below.
%\marginpar{\textcolor{red}{\tiny $\leftarrow$ Proof is missing}}

%\medskip

\subsection{The fractional diffusion equation with an interface} 

Let us now discuss the weak solutions to the fractional diffusion equation with an interface
that will arise as the long time asymptotics of the kinetic interface problem. 
For $\beta>1$, we define the fractional Laplacian $\Lambda_\beta=(-\partial_{y}^2)^{(\beta+1)/(2\beta)}$ 
%$|\partial_{y}^2|^{(\beta+1)/(2\beta)}$ 
as  the $L^2$ (self-adjoint) closure of the 
singular integral operator  %, cf \cite{kwasnicki}, p. 9,
\begin{equation}
\label{frac-lap}
\Lambda_\beta F(y):=%(-\partial_{y}^2)^{(\beta+1)/(2\beta)} F(y):=
{\rm
  p.v.}\int_{\bbR}q_\beta(y-y')\left[F(y)-F(y')\right]dy',\quad F\in C_0^\infty(\bbR),
\end{equation}
understood in the sense of the principal value, with $q_\beta(y)$ as in (\ref{cal}). 
%$$
%q_\beta(y)=\frac{c_\beta}{|y|^{2+1/\beta}}, %\quad y\in\bbR_*,
%$$
%and the normalization constant
%\begin{equation}
%\label{cal-bis}
% c_\beta:=\frac{2^{1+1/\beta}\Gamma\left(1+1/(2\beta)\right)}{\sqrt{\pi}\Gamma\left(-1/2+1/(2\beta)\right)}.
%\end{equation}
%The Fourier transform definition of $\Lambda_\beta$ is
%\begin{equation}
%\label{frac-lap1}
%{\cal F}\left(\Lambda_\beta
%  G\right)(\xi)=(2\pi|\xi|)^{1+1/\beta}{\cal F}(G)(\xi),\quad \xi\in \bbR,
%\end{equation}
%where
%$$
%{\cal F}(G)(\xi)=\int_{\bbR} e^{-2\pi i\xi y}G(y)dy
%$$ 
%is the Fourier transform of a function $G$. 

The operator $-\Lambda_\beta$ is the generator of a L\'evy process. In order to introduce an interface, let us assume that if a particle tries to make a L\'evy jump from $y$ to $y'$ such that $y$ and~$y'$ have the same sign, then the jump happens almost surely. However, if $y$ and $y'$ have
different signs, then with the probability $p_+$ the particle jumps to $y'$, with probability $p_-$ it jumps to~$(-y')$ and with probability 
${\frak g}_0$ it is killed at the interface $y=0$, where a boundary condition $W(t,0)=T$ is prescribed. Recall that these probabilities
satisfy (\ref{012304}). The corresponding Kolmogorov equation is then (\ref{apr204}). 
%\begin{equation}\label{apr204-bis}
%\begin{aligned}
%&\partial_tW(t,y)=\hat c\int_{yy'>0}q_\beta(y-y')[W(t,y')-W(t,y)]dy'+\hat c{\frak g}_0\int_{yy'<0}q_\beta(y-y')[T-W(t,y)]dy'\\
%&+\hat cp_-\int_{yy'<0}q_\beta(y-y')[W(t,-y')-W(t,y)]dy'+
%\hat cp_+\int_{yy'<0}q_\beta(y-y')[W(t,y')-W(t,y)]dy',
%\end{aligned}
%\end{equation}
%where $\hat c$ is a fractional diffusion coefficient. 
Using relation (\ref{012304}), the right side of (\ref{apr204}) can be re-written as
%\begin{equation}\label{apr206}
%\begin{aligned}
%&\partial_tW(t,y)=-\hat c\Lambda_\beta W(t,y)+\hat c\int_{yy'<0}q_\beta(y-y')\big[-W(t,y')+W(t,y)+p_-[W(t,-y')-W(t,y)]\\
%&+ p_+\ [W(t,y')-W(t,y)] 
% + {\frak g}_0\ [T-W(t,y)]\big]dy'\\
% &=-\hat c\Lambda_\beta W(t,y)+\hat c\int_{yy'<0}q_\beta(y-y')\big[{\frak g}_0T+(-1+p_+)W(t,y')+p_-W(t,-y')\big]dy',
%\end{aligned}
%\end{equation}
%which can be written as
\begin{equation}\label{013101-19x}
\partial_tW(t,y)=-\hat c\Lambda_\beta W(t,y)+\hat c\int\limits_{yy'<0}q_\beta(y-y')\big[{\frak g}_0(T-W(t,y')) 
+p_-(W(t,-y')-W(t,y'))\big]dy'.
\end{equation}
% It is convenient to represent $W(t,y)$ in the form
% \begin{equation}\label{apr408}
% W(t,y)=\Psi(t,y)+T\chi(y),
% \end{equation}
% where $\chi(y)$ is a smooth compactly supported function such that
% $\chi(y)=1$ for $|y|\le 1$. 

% The function $\Psi$ satisfies
% \begin{equation}\label{apr416}
% \partial_t\Psi(t,y)=-\hat c\Lambda_\beta \Psi(t,y)+\hat c\int_{yy'<0}q_\beta(y-y')\big[-{\frak g}_0\Psi(t,y')
% +p_-(\Psi(t,-y')-\Psi(t,y'))\big]dy' %+f(t,y),
% \end{equation}
% or, equivalently
% \begin{equation}\label{apr414}
% \begin{aligned}
% &\partial_t\Psi(t,y)=\hat c\int\limits_{yy'>0}q_\beta(y-y')[\Psi(t,y')-\Psi(t,y)]dy'-\hat c{\frak g}_0\Psi(t,y)
% \int\limits_{yy'<0}q_\beta(y-y')dy'\\
% &+\hat cp_-\int\limits_{yy'<0}q_\beta(y-y')[\Psi(t,-y')-\Psi(t,y)]dy'+
% \hat cp_+\int\limits_{yy'<0}q_\beta(y-y')[\Psi(t,y')-\Psi(t,y)]dy'\\
% % &+f(t,y),
% \end{aligned}
% \end{equation}
% with an appropriate $f(t,y)$. 

\begin{df}
\label{df011803-19}
%Suppose that $f\in  L^1_{\rm loc}([0,+\infty),L^2(\bbR))$ and $W_0\in L^2(\bbR)$.
A bounded  function
$W(t,y)$, $(t,y)\in\bar\bbR_+\times \bbR$, is  a  weak
solution to equation \eqref{013101-19x}
% \begin{equation}\label{013101-19xf}
% \begin{aligned}
% &\partial_tW(t,y)=-\hat c\Lambda_\beta W(t,y)\\
% &
% +\hat c\int_{yy'<0}q_\beta(y-y')\big[{\frak g}_0(T-W(t,y')) 
% +p_-(W(t,-y')-W(t,y'))\big]dy' %+f(t,y).
% \end{aligned}
% \end{equation}
if 
 % \begin{itemize}
%  \item[(i)] $W \in L^\infty_{\rm loc}([0,+\infty),L^2(\bbR))$ and
%    $W-T\in  L^2_{\rm loc}([0,+\infty),{\cal H}_0)$ and
% \item[(ii)] 
  for any $t_0>0$ and $G\in C^\infty_0([0,t_0]\times\bbR_*)$ we have
\begin{align} &
\label{021803-19xx} 
0=\int_0^{t_0}dt\int_{\bbR}\left\{\partial_tG(t,y) -\hat c
   %|\partial_{y}^2|^{(\beta+1)/(2\beta)} 
 \Lambda_\beta  G(t,y)\right\} W(t,y)dy %+\int_0^{t_0}dt\int_\Rm f(t,y)G(t,y)dy
\nonumber\\
&
+\hat c  \int_0^{t_0}dt\int_{\bbR} G(t,y)dy\left\{
  p_-\int_{[yy'<0]}q_\beta(y-y')[W(t,-y')-W(t,y')]dy'\right. \\
&
+\left.\frak g_0
  \int_{[yy'<0]}q_\beta(y-y')[T-W(t,y')]dy'\right\}-\int_{\bbR}G(t_0,y)
  W(t_0,y)dy+\int_{\bbR}G(0,y) W_0(y)dy.\nonumber
 \end{align}
% \end{itemize}
\end{df}

Notice that, since the support of the test functions $G$ is bounded away from the interface,
this weak formulation does not give information on the behaviour of the solution at the interface.
In order to capture the behaviour of $W(t,y)$ for $y\to 0^\pm$ we need to consider  solution
in a certain regularity class.
For this purpose we introduce 
the space ${\cal H}_0$ of functions that is the completion of $C_0^\infty(\bbR^d)$
in the norm %such that
\begin{align}
\label{Hnorm}
&\|G\|_{{\cal
  H}_0}^2:=\vphantom{\int_0^1}\sum_{\iota=\pm}\int_{\bbR_\iota^2}q_\beta(y-y')[G(y)-G(y')]^2dydy'
  \vphantom{\int_0^1}+{\frak g}_0
  \int_{\bbR_+\times\bbR_-}q_\beta(y-y')[G^2(y)+ G^2(y')]dydy'
  \nonumber\\
&
\\
&
+\int_{\bbR_+\times\bbR_-}q_\beta(y-y')\left\{p_+[G(y)-G(y')]^2+p_-[G(y)-G(-y')]^2\right\}dydy'
.\nonumber
\end{align}
Note that the term in the last line above, with $p_+$ and $p_-$, is dominated by the
term with the factor of ${\frak g}_0$.

Since $q_\beta(y)\sim |y|^{-2-1/\beta}$, finitness of this norm forces $G(y)$ to decay to $0$ at a certain rate, as $y\to 0$.
% We keep it in the definition of the norm
% so that identity (\ref{apr210}) below holds.  
%The first condition above enforces the regularity of $G$ and the second is a form of the boundary condition $G(0)=0$. 
Let us define the class of function 
\begin{equation}
  \label{eq:1}
  \frak H_T := \{ \exists\; T': \; W - T' \in C([0,+\infty), L^2(\bbR)), W - T \in  L^2_{\rm loc}([0,+\infty),{\cal H}_0)\}
\end{equation}
Clearly if $W\in \frak H_T$, then $W(t,y) \to T$, as $y\to 0$ for almost every $t\ge 0$.

\begin{prop}\label{thm:unique}
A weak solution of \eqref{021803-19xx} is unique in $\frak H_T$. 
\end{prop}

\begin{proof}
In fact, let $\bar W$ be the  difference of two weak solutions in $\frak H_T$. Then $\bar W(t,y)$ is in the space
$$
C([0,+\infty),L^2(\bbR))\cap L^2_{\rm loc}([0,+\infty),{\cal H}_0),
$$
and satisfies \eqref{021803-19xx} for $T=0$.
Approximating $\bar W$ by test functions $G$ in \eqref{021803-19xx} we obtain the identity
\begin{equation}\label{apr210}
\begin{aligned}
&\frac{d}{dt}\|\bar W(t,\cdot)\|_2^2=
-\hat c\|\bar W(t,\cdot)\|_{{\cal H}_0}^2.
\end{aligned}
\end{equation}
%Note that (\ref{apr412}) and the Cauchy-Schwartz inequality allowed us to symmetrize expressions in the third line in the right side above. 
Identity (\ref{apr210}) immediately implies uniqueness of the solutions to \eqref{021803-19xx} in the corresponding space.
\qed
\end{proof}

In Section \ref{sec:proof-theor-refthm01} we prove the following.
\begin{thm}\label{thm012205-19}
  Suppose that $W_0\in {\cal C}_T$ and there exists a
  constant $T'$, so that $W_0-T'\in L^1(\bbR\times \bbT)$. Let $W$
  be the limit of the solutions of the scaled kinetic equation
  described in Theorem $\ref{main-thm}$. Then $W$ belongs to $\frak H_T$.
\end{thm}

\section{Probabilistic representation for a solution to the kinetic
  equation with an interface}\label{sec3}

We now construct a probabilistic interpretation for the kinetic equation with reflection, transmission and absorption at an interface,
as a   generalization of the corresponding jump process without an interface. 
Let $(\Om,{\cal F},\bbP)$ be a probability space
and $\mu$ be a Borel  measure on $\bbT$ given by 
\begin{equation}\label{barR}
\mu(dk)=\frac{R(k)dk}{\bar R},~~
\bar R=\int_{\bbT}R(k)dk.
\end{equation}
We denote by
$(K_n^k)_{n\ge0}$  a Markov chain 
such that $\bbP[K_0^k=k]=1$, with the
transition operator 
\begin{equation}
\label{trp}
Pf(k)=\int_{\bbT}p(k,k')f(k')\mu(dk'),\quad k\in\bbT,\,f\in L^\infty(\bbT).
\end{equation}
Here, $K_n^k$ are the particle momenta between the jump times. 
The measure $\mu$
is ergodic and invariant under $P$, and the operator $P$ can be extended
to~$L^1(\mu)$ and  
\begin{equation}
\label{021102-19}
\|Pf\|_{L^\infty(\mu)}\le \|p \|_{L^\infty}\|f\|_{L^1(\pi)},\quad f\in L^1(\mu).
\end{equation}
The transition operator is  symmetric on $L^2(\mu)$. Since the
transition probability density is strictly positive $\mu\otimes\mu$ a.s.,  
the operator satisfies 
the spectral gap estimate
  \begin{equation}
    \label{eq:8x}
    \sup[\|Pf\|_{L^2(\mu)}:\,f\perp 1,\,\|f\|_{L^2(\mu)}=1]<1.
  \end{equation}
%In addition, since $1$ is a simple eigenvalue for a symmetric,
%negative definite operator $L$ 
We can easily conclude the following  -- see \eqref{L} for the definition
of the operator $L$. %{\bf What is $L$? Is it $P$? How can a negative definite operator have eigenvalue $1$?}
\begin{prop}
\label{prop011102-19}
For any $F\in L^1(\bbT) $ we have
\begin{equation}
\label{etL}
\lim_{t\to+\infty}\|e^{tL}F-\int_{\bbT}F(k)dk\|_{L^1(\bbT)}=0.
\end{equation}
\end{prop}
Next, let $\tau_n$, $n\ge 1$, be a sequence of i.i.d. $\hbox{exp}(1)$
distributed random variables
and $\left(\fT(t)\right)_{t\ge0}$ be a linear interpolation between the times 
\begin{equation}
\label{010502-19}
\fT(n):=\sum_{\ell=0}^{n}\bar t(K_\ell^k)\tau_\ell,\quad n=0,1,\ldots,
\end{equation}
where
\begin{equation}
\label{bart}
\bar t(k):=\frac{1}{\ga_0 R(k)}.
\end{equation}
That is, $\fT(n)$ is the time of the $n$-th jump, and the elapsed times between the consecutive jumps are  $\bar t(K_\ell^k)\tau_\ell$. 
Between the jumps the particle moves with the constant speed $\omega'(K_\ell^k)$, and
the corresponding spatial position
$\tilde Z(t;y,k)$ 
is the linear interpolation  between its locations at the jump times 
\begin{equation}
\label{100401-19}
Z_{n}(y,k):=y-\sum_{\ell=0}^{n}\bar\om'(K_\ell^k)\bar t(K_\ell^k)\tau_\ell,\quad
n=0,1,\ldots 
\end{equation}
%with the convention $\bar\om'(0)\bar t(0)=0$. 
Observe that there exists a
constant $c_+>0$ such that
\begin{equation}
\label{cp}
|\bar\om'(k)\bar t(k)|\le \frac{c_+\om_0'}{\ga_0 R_0|k|^{\beta}},\quad k\in\bbT.
\end{equation}
% In addition\begin{equation}
% P(k,[\bar\om'\bar t\in A])=0,\quad\mbox{for all }k\in \bbT,\mbox{ and }A\in {\cal
%   B}(\bbT)\mbox{ such that }m_1(A)=0,
% \label{non-sing}
% \end{equation}
%\begin{lemma}
%\label{lm1}
Note that for each $n\ge0$ and $(y,k)\in\bbR\times \bbT_*$
the law of $Z_n(y,k)$ is absolutely continuous with respect to the Lebesgue measure on the line. 
%\end{lemma}
%{\bf Proof.} 
%Indeed, this is true for $n=0$, thanks to the absolute
%continuity of the transition probability kernel \eqref{trp}. Suppose that it holds for some $n$. Let a set $A$ have zero Lebesgue measure,
%then
%%for any $A\in{\cal B}(\bbT)$ that has zero Lebesgue measure we have
%\begin{equation}
%\label{null}
%\bbE\left[ 1_A\left(Z_{n+1}(y,k)\right)\right]= 
%%\end{equation}
%%Note that the expression on the left hand side equals
%%\begin{align}
%%\label{062311-18}
%%&
%\bbE\Big[ \int_\bbT
%  1_A\left(Z_n(y,k)+\bar\om'\bar t (k')\right)p(K_n^k,k')\mu(dk')\Big].
%\end{equation}
%However, we have
%$$
%\int_\bbT
%  1_A\left(y+\bar\om'\bar t (k')\right)p(k,k')\mu(dk')=0\quad \mbox{for each $k\in \bbT$
%    and $y\in\bbR$.}
%$$
%This, in turn implies that the expression in \eqref{062311-18} vanishes.
%Thus, \eqref{null} holds. 
%%\qed

We also note that
\begin{equation}
\label{020502-19}
 Y(t;y,k):=\tilde Z\left(\fT^{-1}(t);y,k\right)=y-\int_0^t\bar\om'(K(s;k))ds,\quad K(t;k):=K^k_{[\fT^{-1}(t)]},
\end{equation}
%where $\fT^{-1}(t)$ 
%is the inverse of $\fT(t)$,
%Note that
%$$
%Y(t;y,k):=\tilde Z\left(\fT^{-1}(t);y,k\right),\quad t\ge0.
%$$
and denote by $ {\cal F}_t $ the natural filtration for the process~$ Y(t;y,k), K(t;k)) $.

\subsection*{The jump process with reflection and transmission}

We now add reflection and transmission to the  jump process. 
Suppose that the starting point~$y>0$ and let  %. Let ${\frak s}_{y,1}:=0$ and
\begin{equation}
\label{010501-19a}
{\frak s}_{y,1}:=\inf[n>0:\,Z_n(y,k)<0],\quad 
{\frak s}_{y,2}:=\inf[n>{\frak s}_{y,1}:\, Z_n(y,k)>0]
\end{equation} 
be the first times of the momenta jumps after the first crossing to the left and after the first crossing back to the right. 
Having defined ${\frak s}_{y,2m-1}$, ${\frak s}_{y,2m}$ for some $m\ge1$ we let
\begin{align}
\label{010501-19}
{\frak s}_{y,2m+1}:=\inf[n>{\frak s}_{y,2m}:\,   Z_n(y,k)<0],\quad 
 {\frak s}_{y,2m+2}:=\inf[n>{\frak s}_{y,2m+1}:\, Z_n(y,k)>0].
\end{align}
We define ${\frak s}_{y,m}$ by symmetry also for $y<0$.
We also let
$$
\tilde {\frak s}_{y,1}:=\inf[t>0:\,\tilde Z(t;y,k)<0],\quad 
\tilde {\frak s}_{y,2}:=\inf[t>{\frak s}_{y,1}:\, \tilde Z(t;y,k)>0]
$$  
be the times when the trajectory crosses to the left and then crosses back to the right. 
Having defined $\tilde {\frak s}_{y,2m-1}$, $\tilde {\frak s}_{y,2m}$ for some $m\ge1$, we set
\begin{align}
\label{010501-19am}
\tilde{\frak s}_{y,2m+1}:=\inf[t>\tilde {\frak s}_{y,2m}:\,  \tilde Z(t;y,k)<0],\quad 
 \tilde {\frak s}_{y,2m+2}:=\inf[t>\tilde{\frak s}_{y,2m+1}:\, \tilde Z(t;y,k)>0].
\end{align}
and. again, by symmetry we define $\tilde{\frak s}_{y,m} $
for $y<0$.
Obviously, we have 
$$
\tilde{\frak s}_{y,m}<{\frak s}_{y,m}<\tilde{\frak s}_{y,m+1},\quad \mbox{a.s.}
$$

We let $\hat\si_m^{y}$ be a $\{-1,0,1\}$-valued sequence of random
variables that are independent, when conditioned on $(K_n^k)_{n\ge0}$,
such that
$$
\bbP[\hat\si_m^y=0|(K_n^k)_{n\ge0}]= \frak g(K^k_{ {\frak s}_{y,m}-1}),\quad \bbP[\hat\si_m^y=\pm1|(K_n^k)_{n\ge0}]=p_\pm(K_{{\frak s}_{y,m}-1}^k).
$$
% $$
% \tilde\si_m^y=\left\{
% \begin{array}{cl}
% 1,& \mbox{when }\hat\si_m^y=0,1,\\
% &\\
% -1,& \mbox{when }\hat\si_m^y=-1
% \end{array}
% \right.
% $$
% and
 These variables are responsible for whether the particle is reflected, transmitted or absorbed as it crosses the interface, and 
 \begin{equation}
\label{frakf}
{\frak f}:=\inf\big[m\ge 1:\,\hat \sigma^y_m=0\big]
\end{equation}
 is
the crossing at which the particle is absorbed. 
For $m\ge1$, we denote by ${\frak F}_m$ the $\si$-algebra generated by $(Y(t;y,k),
K(t;k))$, $0\le t\le \tilde {\frak s}_{y,m}$ and $\hat \si_\ell^y$,
$\ell=1,\ldots,m$, with the convention that~${\frak F}_0$ is the trivial $\si$-algebra.
Recall that $\left({\cal F}_t\right)_{t\ge0}$ is the natural filtration
corresponding to the process  $\left(Y(t;y,k),
K(t;k)\right)_{t\ge0}$.

We define the reflected-transmitted-absorbed process
$$
\tilde Z^o(t;y,k):=\Big(\prod_{m=1}^{n-1}\hat\si_m^y\Big) \tilde
Z(t;y,k),\quad t\in[\tilde{\frak s}_{y,n-1},\tilde{\frak s}_{y,n}),
$$
with the convention that the product over an empty set of indices
equals $1$
% and
%  the chain
% $$
% K_{\ell}^{k,o}:=\left(\prod_{m=1}^{n-1}\hat\si_m^y\right)
% K_{\ell}^{k},\quad \ell\in\{{\frak s}_{y,n-1},\ldots,{\frak s}_{y,n}-1\}
% $$
and the respective counterparts
$$
 Y^o(t;y,k):=\tilde
 Z^o(T^{-1}(t);y,k)=\Big(\prod_{m=1}^{n-1}\hat\si_m^y\Big)
 Y(t;y,k),\quad  t\in[\tilde{\frak s}_{y,n-1},\tilde{\frak s}_{y,n})
$$
and
$$
K^o(t;k) :=\Big(\prod_{m=1}^{n-1}\hat\si_m^y\Big) K(t;k),\quad t\in[\tilde{\frak s}_{y,n-1},\tilde{\frak s}_{y,n}).
$$
In what follows, we assume the convention that $\bar\om'(0):=0$ even though $\omega(k)$
is not differentiable at $k=0$.
For $t\in (\tilde{\frak s}_{y,n-1},\tilde{\frak s}_{y,n})$ we can
write  
\begin{align}
\label{040902-19}
&
\frac{d Y^o(t;y,k)}{dt}:=\Big(\prod_{m=1}^{n-1}\hat\si_m^y\Big) 
\frac{d Y(t;y,k)}{dt}
=-\Big(\prod_{m=1}^{n-1}\hat\si_m^y\Big) 
\om'(K(t;k))\nonumber\\
&
=-\om'\Big(\Big(\prod_{m=1}^{n-1}\hat\si_m^y\Big) K(t;k)\Big)=-\om'( K^o(t;k)).
\end{align}
 % Note that
% \begin{align*}
% %&(Y^o(t;y,k),
% %K^o(t;k))=(Y(t;y,k),
% %K(t;k)),\quad t\in[\tilde {\frak s}_{y,0},\tilde {\frak s}_{y,1}),\\
% &
% Y^o(t;y,k)=\left(\prod_{m=1}^{n-1}\hat\si_m^y\right) Y(t;y,k),\\
% &
% K^o(t;k)=\left(\prod_{m=1}^{n-1}\hat\si_m^y\right)K(t;k),\quad t\in[\tilde {\frak s}_{y,n-1},\tilde {\frak s}_{y,n}),\quad n=1,2,\ldots.
% \end{align*}
If $Y^o(t;y,k)\ge 0$ for  $\tilde {\frak s}_{y,m-1}< t<\tilde
{\frak s}_{y,m+1}$ -- that is, the particle approached the interface from the right at the time $\tilde {\frak s}_{y,m}$ 
and was reflected, then 
 for $h>0$ we define $\tilde {\frak s}_{y,m}^h\in(\tilde {\frak s}_{y,m-1},\tilde{\frak s}_{y,m})$ 
 as the first exit  time of $Y^o(t;y,k)$ from the
half-line $[y>h]$ that happens after~$\tilde {\frak s}_{y,m-1}$, 
and~$\tilde {\frak s}_{y,m}^{h,e}\in(\tilde {\frak s}_{y,m},\tilde{\frak s}_{y,m+1})$
as the  first exit  time of $Y^o(t;y,k)$ from the
half-line $[y<h]$, after~$\tilde {\frak s}_{y,m}$. Note that both $\tilde {\frak s}_{y,m}^h$ and $\tilde {\frak s}_{y,m}^{h,e}$
are finite a.s. if $h>0$ is sufficiently small, and we have
%Since  $Y^o(t;y,k)$ for  $\tilde {\frak s}_{y,m-1}\le  t<\tilde
%{\frak s}_{y,m+1}$ is piecewise linear we have
$$
\lim_{h\to0+}\tilde{\frak s}_{y,m}^h=\lim_{h\to0+}\tilde{\frak s}_{y,m}^{h,e}=\tilde{\frak s}_{y,m},\quad {\rm a.s.}
$$
Analogous definitions can be introduced for all other configurations of the signs of~$Y^o(t;y,k)$
in~$\tilde {\frak s}_{y,m-1}< t<\tilde
{\frak s}_{y,m+1}$.

\subsection*{A probabilistic representation for  the kinetic equation} %(\ref{eq:8})}

We will now prove the following. 
\begin{prop}
\label{prop010701-19a}
If $ W(t,y,k)$ is a solution to \eqref{eq:8}, with
the interface condition~\eqref{feb1408}-\eqref{feb1410}, in the sense of Definition $\ref{df013001-19}$, such that
$ W(0,y,k) =W_0(y,k)$ and  $ W_0\in {\cal C}_T$, then 
\begin{align}
\label{010701-19a}
& W(t,y,k)=\bbE\left[W_0\left(Y^o( t;y,k), K^o( t;k)\right),\,
  t< \tilde  {\frak s}_{y,{\frak f}}\right]+T\bbP\left[
  t\ge \tilde  {\frak s}_{y,{\frak f}}\right],~~t\ge 0,~y\in\Rm_*,~k\in\Tm_*. %\quad (t,y,k)\in\bar\bbR_+\times\bbR_*\times\bbT_*.
\end{align}
\end{prop}
\proof
First, let $\widetilde W(t,y,k)$ be a solution to \eqref{eq:8} 
as in Proposition~\ref{prop010701-19a} but with $T=0$ in
the interface conditions \eqref{feb1408}-\eqref{feb1410}.
We set
\begin{align}
\label{011303-19}
&{\cal M}_m:= %\sum_{j=0}^{m-1}Z_j=
\lim_{h\to0+} 1_{[{\frak f}>m-1]} \widetilde W\left( t-t\wedge \tilde  {\frak s}_{y,m+1}^h
,Y^o( t\wedge\tilde  {\frak s}_{y,m+1}^h;y,k), K^o( t\wedge\tilde  {\frak
  s}_{y,m+1}^h;k)\right)\nonumber\\
&
+1_{[{\frak f}\le m-1]} \widetilde W_0\left( Y^o( t;y,k), K^o( t;k)\right)-W(t,y,k)
\end{align}
and consider the increments
\begin{equation}\label{apr1504} 
\begin{aligned} 
&{\cal Z}_m:=\lim_{h\to0+}\left\{ \widetilde W\left( t-t\wedge \tilde  {\frak s}_{y,m+1}^h
,Y^o( t\wedge\tilde  {\frak s}_{y,m+1}^h;y,k), K^o( t\wedge\tilde  {\frak
  s}_{y,m+1}^h;k)\right)\right.\\
&
\left. -\widetilde W\left( t- t\wedge\tilde  {\frak s}_{y,m}^h
,Y^o(t\wedge\tilde  {\frak s}_{y,m}^h;y,k), K^o(t\wedge \tilde  {\frak
  s}_{y,m}^h;k)\right)\right\},\quad m=0,\ldots,{\frak
  f}-1,
\\
&{\cal Z}_{\frak f}:=-\lim_{h\to0+}\widetilde W\left( t-\tilde  {\frak s}_{y, {\frak f}}^h
,Y^o(\tilde  {\frak s}_{y, {\frak f}}^h;y,k), K^o(\tilde  {\frak
  s}_{y, {\frak f}}^h;k)\right) \mbox{ on the event } [{\frak s}_{y, {\frak
  f}}\le t],\\
&{\cal Z}_{\frak f}=0,\hbox{ on the event $[{\frak s}_{y, {\frak
  f}}>t]$},\\
  &{\cal Z}_m:=0, ~~m> {\frak f},
\end{aligned}
\end{equation}
so that
\begin{equation}\label{apr902}
{\cal M}_m:=\sum_{j=0}^{m-1}{\cal Z}_j.
\end{equation}
%and ${\cal Z}_{\frak f}=0$, if ${\frak s}_{y, {\frak
%  f}}>t$
%and ${\cal Z}_m:=0$,  $m> {\frak f}.$
The key step in the proof of Proposition~\ref{prop010701-19a} is the following lemma.
\begin{lemma}
\label{prop010501-19}
We have
\begin{equation}
\label{100501-19}
\bbE[{\cal Z}_m|{\frak F}_m]=0,\quad m=0,1,\ldots.
\end{equation}
\end{lemma}
As an immediate corollary of Lemma \ref{prop010501-19}, we know that
%\begin{cor}
%\label{cor011303-19}
the sequence $\left({\cal M}_m\right)_{m\ge1}$ is a zero mean martingale with respect
to the filtration~$\left({\frak F}_{m}\right)_{m\ge1}$. Since
${\frak f}+1$ is a stopping time with respect to the
filtration  $\left({\frak F}_{m}\right)_{m\ge1}$, and the
martingale $\left({\cal M}_m\right)_{m\ge1}$ is bounded, the optional stopping
theorem implies that~$\bbE{\cal M}_{{\frak f}+1}=0$, which yields %\eqref{010701-19}.
\begin{align}
\label{010701-19}
&\widetilde W(t,y,k)=\bbE\left[W_0\left(Y^o( t;y,k), K^o( t;k)\right),\,
  t< \tilde  {\frak s}_{y,{\frak f}}\right],~~t>0,~y\in\Rm_*,~k\in\Tm_*,
%  \quad (t,y,k)\in\bar\bbR_+\times\bbR_*\times\bbT_*..
\end{align}
which is a special case of (\ref{010701-19a}) with $T=0$. 
%
%%\end{cor}
%\begin{cor}
%\label{cor010701-19}
%If $\widetilde W(t,y,k)$ is a solution to \eqref{eq:8}, with
%the interface condition~\eqref{feb1408}-\eqref{feb1410}
%with $T=0$, in the sense of Definition $\ref{df013001-19}$, and that
%$\widetilde W(0,y,k) =W_0(y,k)$ is in $ {\cal C}_0$, then
%\begin{align}
%\label{010701-19}
%&\widetilde W(t,y,k)=\bbE\left[W_0\left(Y^o( t;y,k), K^o( t;k)\right),\,
%  t< \tilde  {\frak s}_{y,{\frak f}}\right],~~t>0,~y\in\Rm_*,~k\in\Tm_*.
%%  \quad (t,y,k)\in\bar\bbR_+\times\bbR_*\times\bbT_*..
%\end{align}
%\end{cor}
%{\bf Proof.}  
%
%\bigskip

In general, if $W(t,y,k)$ is as in Proposition~\ref{prop010701-19a}, with $T\neq 0$,
then $\widetilde W(t,x,k)=W(t,y,k) -T $ satisfies \eqref{eq:8}, with
the interface condition given by \eqref{feb1408} and \eqref{feb1410}
corresponding to $T=0$ and the initial condition $\widetilde
W_0(y,k)=W_0(y,k)-T$. It follows from the above that
%\begin{cor}
%\label{cor010701-19a}
%Suppose that $ W(t,y,k)$ is a solution of \eqref{eq:8}, with
%the interface condition given by \eqref{feb1408} and \eqref{feb1410}
%for some $T>0$, in the sense of Definition $\ref{df013001-19}$. Assume that
%$ W(0,y,k) =W_0(y,k)$ and  $ W_0\in {\cal C}_T$. Then,
%\begin{align}
%\label{010701-19a}
%& W(t,y,k)=\bbE\left[W_0\left(Y^o( t;y,k), K^o( t;k)\right),\,
%  t< \tilde  {\frak s}_{y,{\frak f}}\right]+T\bbP\left[
%  t\ge \tilde  {\frak s}_{y,{\frak f}}\right],\quad (t,y,k)\in\bar\bbR_+\times\bbR_*\times\bbT_*.
%\end{align}
%\end{cor}
%{\bf Proof.} Note that $ \widetilde W(t,y,k)= W(t,y,k) -T $ satisfies \eqref{eq:8}, with
%the interface condition given by \eqref{feb1408} and \eqref{feb1410}
%corresponding to $T=0$ and the initial condition $\widetilde
%W_0(y,k)=W_0(y,k)-T$.  According to Corollary \eqref{cor010701-19} we have
\begin{align}
\label{010701-19aa}
& W(t,y,k)=T+\widetilde W(t,y,k)=T+\bbE\left[\widetilde W_0\left(Y^o( t;y,k), K^o( t)\right),\,
  t< \tilde  {\frak s}_{y,{\frak f}}\right]\nonumber\\
&
=\bbE\left[W_0\left(Y^o( t;y,k), K^o( t)\right),\,
  t< \tilde  {\frak s}_{y,{\frak f}}\right]+T\bbP\left[
  t\ge \tilde  {\frak s}_{y,{\frak f}}\right]
\end{align}
and \eqref{010701-19a} follows, finishing the proof of Proposition~\ref{prop010701-19a}.
\qed

%\bigskip
%
%The above result in particular implies also the following.
%\begin{cor}
%\label{cor013101-19}
%Equation \eqref{eq:8} with the interface conditions
%\eqref{feb1408} and \eqref{feb1410} admits at most one solution in the sense of Definition $\ref{df013001-19}$. 
%\end{cor}
%
%
%\bigskip
%
%

{\bf Proof of Lemma \ref{prop010501-19}.}
Let ${\widetilde W}_\pm(t,y,k)$ be the restrictions of 
$\widetilde W$ to  $\{y>0\}$ and $\{y<0\}$, respectively. %$\bbR_+\times \bbR_\pm\times\bbT_*$. 
We extend them to the whole line
%$\bbR_+\times \bbR\times\bbT_*$ 
in such a way that $D_t\widetilde W_\pm$ are well defined
for all $(t,y,k)\in \bbR_+\times \bbR\times\bbT_*$ and they are
bounded and measurable, and 
denote
$$
F_\pm(t,y,k):=(L_k-\om'(k)\partial_y-\partial_t)\widetilde W_\pm(t,y,k).
$$
Note that 
$
F_\pm(t,y,k)=0$ for $y\in\Rm_{\pm}$, respectively, 
%\quad\mbox{provided that }(t,y,k)\in \bbR_+\times \bbR_\pm\times\bbT_*.$
and the processes
$$
{\cal M}_\pm(u):=\widetilde W_\pm(t-u,Y(u;y,k),K(u;k))-\widetilde W_\pm (t,y,k)-\int^u_0 F_\pm(t-s,Y(s;y,k),K(s;k))ds,
$$
with $0\le u\le t$  are  ${\cal F}_u $-martingales, so that
\begin{align}
\label{020701-19}
&
\bbE\left\{\left[\vphantom{\Big|}\widetilde W_{\si'}\left(t-t\wedge \tilde{\frak s}_{y,m+1}^h,\si Y(t\wedge \tilde{\frak s}_{y,m+1}^h;y,k),\si K(t\wedge \tilde{\frak s}_{y,m+1}^h;k))\right)\right.\right.\nonumber\\
&
\\
&
\left.\left.-\vphantom{\Big|}\widetilde W_{\si'}\left(t-t\wedge \tilde{\frak s}_{y,m}^{h,e},\si Y(t\wedge \tilde{\frak s}_{y,m}^{h,e};y,k),\si K(t\wedge \tilde{\frak s}_{y,m}^{h,e};k))\right)\right]\Big|{\cal F}_{\tilde {\frak s}_{y,m}}\right\}=0,\nonumber
\end{align}
{provided that 
$$
\si':=(-1)^{m}\si {\rm sign}\,y.
$$
Note that since 
$$
(-1)^{m} {\rm sign}\,y={\rm sign} \left(Y(t\wedge \tilde{\frak s}_{y,m}^{h,e};y,k)\right)
$$
we have
$$
\si'={\rm sign} \left(\si Y(t\wedge \tilde{\frak s}_{y,m}^{h,e};y,k)\right).
$$}

The interface conditions \eqref{feb1408} and \eqref{feb1410} with $T=0$ can be written
as
\begin{align}
\label{040701-19}
& p_+(\si K(\tilde  {\frak
  s}_{y,m};k))\widetilde W_{-\si'}\left( t- \tilde  {\frak s}_{y,m}
,0, \si K( \tilde  {\frak
  s}_{y,m};k)\right)\\
&
+p_-(\si  K(\tilde  {\frak
  s}_{y,m};k))\widetilde W_{\si'}\left( t- \tilde  {\frak s}_{y,m}
,0, -\si K( \tilde  {\frak
  s}_{y,m};k)\right) 
  =\widetilde W_{\si'}\left( t- \tilde  {\frak s}_{y,m}
,0, \si K( \tilde  {\frak
  s}_{y,m};k)\right), \nonumber
\end{align}
provided that
\begin{equation}
\label{010801-19}
\si'\si  {\rm sgn}\,K(\tilde  {\frak
  s}_{y,m};k)=-1.
\end{equation}

We now need to replace the time $\tilde{\frak s}_{y,m}^{h,e}$ in the second term
in (\ref{020701-19}) by $\tilde{\frak s}_{y,m}^{h}$, in order to convert 
the right side of~(\ref{020701-19})
into a term of a telescoping sum, and to show that ${\cal M}_m$ is a martingale.
To this end, suppose that $\Phi\in C_b((\bbR\times \bbT)^{L+1}\times\{-1,0,1\}^{m-1})$ 
and consider the times $0=t_0<t_1<\ldots<t_L$. Then, we have
\begin{align}\label{apr1502}
&
\bbE\left[{\cal Z}_m\Phi_m\right]=\sum_{\eps}\bbE\left[{\cal Z}_m\Phi_m,\,A_{\eps_1,\ldots,\eps_{m-1}}\right],
\end{align}
with
$$
\Phi_m:=\Phi\left( \left(Y(t_j\wedge \tilde{\frak
      s}_{y,m};y,k),K(t_j\wedge \tilde{\frak s}_{y,m};k)\right)_{0\le
    j\le L},\hat \si_1^y,\ldots,\hat \si_{m-1}^y\right)
$$
$$
A_{\eps_1,\ldots,\eps_{m-1}}=[\hat \si_j^y=\eps_j,\,j=1,\ldots,m-1],
$$
and the summation in (\ref{apr1502})
extending over all sequences $\eps=(\eps_1,\ldots,\eps_{m-1})\in\{-1,0,1\}^{m-1}$.
Suppose that some $\eps_{j}=0$. Then, we have
$$
{\frak f}=\min[j\in\{ 1,\ldots,m\}:\,\eps_j=0]\le m-1
$$ and  ${\cal Z}_m1_{A_\eps}=0$ for the corresponding sequence $(\eps_1,\ldots,\eps_{m-1})$.
On the other hand, if $\eps_j\not=0$ for all~$j=1,\ldots,m-1$, then 
$$
\bbE[{\cal Z}_m\Phi_m, A_\eps]={\cal I}_++{\cal I}_0+{\cal I}_-,
$$
  where ${\cal I}_\iota$, $\iota\in\{-1,0,1\}$, correspond to the
  events $\hat\si^y_{m}=\iota$. 
% Therefore
% \begin{align*}
% &\bbE[{\cal Z}_m\Phi_m, A_\eps]={\cal I}_++{\cal I}_-.
%  % \\
% % &
% % =\sum_{\iota=\pm}\lim_{h\to 0+}\bbE\left[\left\{ \widetilde W\left( t-t\wedge \tilde  {\frak s}_{y,m+1}^h
% % ,Y^o( t\wedge\tilde  {\frak s}_{y,m+1}^h;y,k), K^o( t\wedge\tilde  {\frak
% %   s}_{y,m+1}^h;k)\right)\right. \right.\\
% % &
% % \left. \left.-\widetilde W\left( t- \tilde  {\frak s}_{y,m}^h
% % ,Y^o( \tilde  {\frak s}_{y,m}^h;y,k), K^o(\tilde  {\frak
% %   s}_{y,m}^h;k)\right)\right\}\Phi_m, A_\eps,\,t\ge \tilde  {\frak s}_{y,m}^h,\, \hat\si^y_{m}=\iota\right].
%   \end{align*}
 Knowing the values $\hat \si^y_1,\ldots,\hat\si^y_{m-1}$ and the sign of $y$ one can determine
the sign $\si$ in the equality 
$$
Y^o( t;y,k)=\si Y( t;y,k), \quad K^o( t;k)=\si K( t;k),\quad  \tilde  {\frak s}_{y,m}^h\le t< \tilde  {\frak s}_{y,m},
$$
hence
$$
\widetilde W\left( t- \tilde  {\frak s}_{y,m}^h
,Y^o( \tilde  {\frak s}_{y,m}^h;y,k), K^o(\tilde  {\frak
  s}_{y,m}^h;k)\right)=\widetilde W_{\si'}\left( t- \tilde  {\frak s}_{y,m}^h
,\si Y( \tilde  {\frak s}_{y,m}^h;y,k), \si K(\tilde  {\frak
  s}_{y,m}^h;k)\right),
$$
with {
$
\si':=(-1)^{m}\si {\rm sign}\,y.
$}
We have 
% \begin{align*}
% &{\cal I}_0 =\lim_{h\to 0+}\bbE\left\{\left\{ T-W_{\si'}\left( t- \tilde  {\frak s}_{y,m}^h
% ,\si Y( \tilde  {\frak s}_{y,m}^h;y,k), \si K(\tilde  {\frak
%   s}_{y,m}^h;k)\right)\right\}\Phi_m, A_\eps,\,\hat\si^y_{m}=0,\,t\ge \tilde  {\frak s}_{y,m}^h\right\}\\
%   &
%   = \bbE\left\{\left\{ T{\frak g}(\si K(\tilde  {\frak
%   s}_{y,m};k))-W_{\si'}\left( t- \tilde  {\frak s}_{y,m}
% ,0, \si K(\tilde  {\frak
%   s}_{y,m};k)\right)1_{[\hat\si^y_{m}= 0]}\right\}\Phi_m, A_\eps,\,t\ge \tilde  {\frak s}_{y,m}\right\}
%   \end{align*}
%   and
\begin{align*}
&{\cal I}_\pm  =\lim_{h\to 0+}\bbE\left\{\left\{ \widetilde W_{\mp\si'}\left( t-t\wedge \tilde  {\frak s}_{y,m+1}^h
,\pm\si Y( t\wedge\tilde  {\frak s}_{y,m+1}^h;y,k), \pm\si K( t\wedge\tilde  {\frak
  s}_{y,m+1}^h;k)\right)\right.\right.\\
&
\left. \left.-\widetilde W_{\si'}\left( t- \tilde  {\frak s}_{y,m}^h
,\si Y( \tilde  {\frak s}_{y,m}^h;y,k), \si K(\tilde  {\frak
  s}_{y,m}^h;k)\right)\right\}\Phi_m,
  A_\eps,\,\hat\si^y_{m}=\pm1,\,t\ge \tilde  {\frak
  s}_{y,m}^h\right\}\\
  &
  =\lim_{h\to 0+}\bbE\left\{\left\{ \widetilde W_{\mp\si'}\left( t-t\wedge \tilde  {\frak s}_{y,m+1}^h
,\mp\si Y( t\wedge\tilde  {\frak s}_{y,m+1}^h;y,k), \mp\si K( t\wedge\tilde  {\frak
  s}_{y,m+1}^h;k)\right)p_\pm(\si K(\tilde  {\frak
  s}_{y,m};k))\right.\right.\\
&
\left. \left.-\widetilde W_{\si'}\left( t- \tilde  {\frak s}_{y,m}
,\si Y( \tilde  {\frak s}_{y,m};y,k), \si K(\tilde  {\frak
  s}_{y,m};k)\right)1_{[\hat\si^y_{m}=\pm1]}\right\}\Phi_m, A_\eps,\,t\ge \tilde  {\frak s}_{y,m}\right\}.
  \end{align*}
Passing to the limit $h\to 0+$ above, using the continuity of $\widetilde W_\pm$ up to the
interface and the fact that $ {\frak s}_{y,m}^{h,e}\to {\frak s}_{y,m} $ 
as $h\to 0+$ a.s., and
invoking \eqref{020701-19}, we conclude that
\begin{align*}
&{\cal I}_\pm  =\bbE\left\{\left\{\widetilde W_{\mp\si'}\left( t- \tilde  {\frak s}_{y,m}
,\pm\si Y( \tilde  {\frak s}_{y,m};y,k), \pm\si K( \tilde  {\frak
  s}_{y,m};k)\right)p_\pm(K(\tilde  {\frak
  s}_{y,m};k))\right.\right.\\
&
\left. \left.-\widetilde W_{\si'}\left( t- \tilde  {\frak s}_{y,m}
,\si Y( \tilde  {\frak s}_{y,m};y,k), \si K(\tilde  {\frak
  s}_{y,m};k)\right)1_{[\hat\si^y_{m}=\pm1]}\right\}\Phi_m, A_\eps,\,t\ge \tilde  {\frak s}_{y,m}\right\}.
  \end{align*}
On the event $[\hat\si^y_{m}=0]$
we have ${\frak f}=m$, therefore 
\begin{align*}
&{\cal I}_0 =-\bbE\left\{\widetilde W_{\si'}\left( t- \tilde  {\frak s}_{y, m}
,\si Y( \tilde  {\frak s}_{y, m};y,k), \si K(\tilde  {\frak
  s}_{y, m};k)\right)\Phi_m, A_\eps,\,\hat\si^y_{m}=0,\,t\ge \tilde  {\frak s}_{y, m}\right\},
  \end{align*}
as follows from the condition ${\cal Z}_{\frak f}=0$ on the event 
$[{\frak s}_{y, {\frak f}}>t]$ in (\ref{apr1504}).
%Here, $\si'$ and $\si K(\tilde  {\frak
%  s}_{y,m};k)$ satisfy \eqref{010801-19}.
Now, we conclude that from  \eqref{040701-19} that
$$
{\cal I}_++{\cal I}_0+{\cal I}_-=0,
$$
thus \eqref{100501-19} follows.
\qed
%
%Let us now set
%\begin{align}
%\label{011303-19}
%&{\cal M}_m:=\sum_{j=0}^{m-1}Z_j=\lim_{h\to0+} 1_{[{\frak f}>m-1]} \widetilde W\left( t-t\wedge \tilde  {\frak s}_{y,m+1}^h
%,Y^o( t\wedge\tilde  {\frak s}_{y,m+1}^h;y,k), K^o( t\wedge\tilde  {\frak
%  s}_{y,m+1}^h;k)\right)\nonumber\\
%&
%+1_{[{\frak f}\le m-1]} \widetilde W_0\left( Y^o( t;y,k), K^o( t;k)\right)-W(t,y,k)
%\end{align}

% \textcolor{red}{\em write the probabilistic representation for arbitrary temperature $T>0$}

\section{The scaled processes and their convergence}
\label{sec4}

\subsection{Convergence of processes without an interface}

We consider the rescaled process 
$$
Z_N(t;y,k):=y-\frac{1}{N^{\beta/(1+\beta)}}\sum_{\ell=0}^{[Nt]}
\bar\om'(K_\ell^k)\bar t(K_\ell^k)\tau_\ell,\quad t\ge0,
$$
and  $\tilde Z_N(t;y,k)$ be the linear interpolation in time between the
points $Z_N(n/N;y,k)$,~$n\ge 0$.  
Let also $\left({\frak T}_N(t;k)\right)_{t\ge0}$ and
$\left(Y_N(t;y,k)\right)_{t\ge0}$ be the
scaled versions of the processes defined by~\eqref{010502-19} and \eqref{020502-19},
respectively:
\begin{equation}
\label{010502-19N}
\fT_N(t;k):=\frac{1}{N}\sum_{\ell=0}^{m}\bar
t(K_\ell^k)\tau_\ell,\quad t=\frac{m}{N}
\end{equation}
and it is a linear interpolation otherwise, while
\begin{equation}
\label{020502-19N}
 Y_N(t;y,k):=y-\frac{1}{N^{\beta/(1+\beta)}}\int_0^{Nt}\bar\om'(K(s;k))ds=\tilde
 Z_N(\fT_N^{-1}(t,k),y,k).
\end{equation}
To describe the limit, let
$\left(
\zeta(t)\right)_{t\ge0}
$ be the symmetric stable process with the L\'{e}vy exponent  
$$
\psi(\theta)=\frac{\hat c |\theta|^{1+1/\beta}}{\ga_0\bar R},\quad \theta\in\bbR,
$$
and set
\begin{equation}
\label{tau}
\zeta(t,y):=y+\zeta(t),~~~
\tau(t):=t \bar \tau,~~~\eta(t,y):=\zeta(\tau^{-1}(t),y),\quad t\ge0,
\end{equation}
where 
\begin{equation}
\label{bT}
\bar \tau:=\int_{\bbT} \bar t(k)\mu(dk)=\frac{1}{\ga_0\bar R}.
\end{equation}
%We define
%$$
%\eta(t,y):=\zeta(\tau^{-1}(t),y),\quad t\ge0.
%$$
\begin{prop}
\label{prop010801-19}
For any $t_0>0$ and $k\in \bbT$ we have
\begin{equation}
\label{020801-19}
\lim_{N\to+\infty}\bbE\Big\{\sup_{t\in[0,t_0]}\left|\fT_N(t;k)-\tau(t)\right|\Big\}=0.
\end{equation}
\end{prop}
The proof of the proposition is standard an we omit it. %presented  in Appendix \ref{app2}.

The following result is a simpler version of 
Proposition \ref{prop010801-19} and   
Theorem 2.5(i) of~\cite{jk}, see also  Theorem 2.4
of \cite{jko} and Theorem 3.2 of \cite{bb}. 
\begin{prop}
\label{joint-conv-2}
Suppose that $\beta>1$ and $(y,k)\in\bbR_*\times\bbT_*$. Under the assumptions 
on the functions~$R $ and  $\om $ in Section $\ref{sec2}$, the joint law of
$\left(Z_N(t,y,k),\fT_N(t;k)\right)_{t\ge0}$ converges in law, over
${\cal D}_2:=D([0,+\infty);\bbR\times\bar\bbR_+)$ equipped with the Skorokhod 
$J_1$-topology to  $\left(\zeta(t,y),\tau(t)\right)_{t\ge0}$.
\end{prop}
The following result is an immediate consequence of the above
theorem.
\begin{cor}
\label{cor010502-19}
The process $\left(Y_N(t,y,k)\right)_{t\ge0}$  converge in law, as
 $N\to+\infty$,
over
$D[0,+\infty)$ equipped with the Skorokhod $M_1$-topology
to  $\left(\eta(t,y,k)\right)_{t\ge0}$.
\end{cor}

\subsection{Joint convergence of processes and crossing times and positions}

Using the analogues of \eqref{010501-19}-\eqref{010501-19am}    
we can define crossing times 
${\frak s}_{y,m}^N$, $\tilde {\frak s}_{y,m}^N$, $m,N=1,2,\ldots$ for
the scaled process
$(Z_N(t,y,k))_{t\ge0}$ and $(\tilde Z_N(t;y,k))_{t\ge0}$, respectively.
As a simple consequence of absolute continuity of the law of $Z_n(y,k)$  we conclude that for each $y\in\bbR$ there exists  a strictly increasing sequence
$(n_{y,m}^N)_{m\ge1}$ such that 
\begin{equation}
\label{010301-19}
 \tilde{\frak
  s}_{y,m}^{N}\le  {\frak s}_{y,m}^{N}=\frac{n_{y,m}^N}{N}< \tilde{\frak
  s}_{y,m}^{N}+\frac1N,\quad \mbox{ a.s.}
\end{equation}
Likewise, let ${\frak u}_{y,m}$ be the consecutive times when the process
$\left(\zeta(t,y)\right)_{t\ge0}$ crosses the level $z=0$.
The main result of this section is the following.
\begin{thm}
\label{thm013112}
For any $(y,k)\in\bbR\times\bbT_*$ the random elements
$$
\Big(\left(Z_{N}(t,y,k), \fT_N(t;k)\right)_{t\ge0},({\frak s}_{y,m}^N)_{m\ge1},(Z_N({\frak
  s}_{y,m}^{N},y,k))_{m\ge1}\Big)
$$ converge in law, as $N\to+\infty$, over
${\cal D}_2\times \bar\bbR^{\bbN}_+\times \bbR^{\bbN}$ with the product of the $J_1$ and standard
product topology on $(\bbR^{\bbN})^2$, to  
$$
\Big((\zeta(t,y),\tau(t))_{t\ge0},({\frak
  u}_{y,m})_{m\ge1},(\zeta({\frak u}_{y,m},y,k)_{m\ge1}\Big).
$$ 
\end{thm}
The proof of this result is contained in Appendix \ref{appA}.

We now formulate a property of the approximating
process $(Z_N(t))_{t\ge0}$ at the crossing times. We start with the
following simple consequence of Corollary 2.2  of \cite{millar} and the
strong Markov property of stable processes.
\begin{lemma}
\label{lm010201-19} %Suppose that $M$ is a positive integer. Then, 
For
each $y>0$ we have
\begin{equation}
\label{050201-19}
\bbP\left[\zeta({\frak s}_{y,2m},y)>0>\zeta({\frak s}_{y,2m-1},y),\,m\ge 1\right]=1.
\end{equation}
If, on the other hand 
 $y<0$, then 
\begin{equation}
\label{050201-19n}
\bbP\left[\zeta({\frak s}_{y,2m-1,y},y)>0>\zeta({\frak s}_{y,2m},y),\,m\ge 1\right]=1.
\end{equation}
\end{lemma}
%{\bf Proof.}
%Thanks to the aforementioned Corollary 2.2  of
%\cite{millar}, we have
%\begin{equation}
%\label{050201-19a}
%\bbP\left[\zeta({\frak s}_{y,1},y)<0\right]=1.
%\end{equation}
%Using the notation of \eqref{tildeZ} and the strong Markov property we can write 
%\begin{align}
%\label{050201-19b}
%&
%\bbP\left[\zeta({\frak u}_{y,1},y)<0, \zeta({\frak
%    u}_{y,2},y)>0\right]=\bbP\left[\zeta({\frak u}_{y,1},y)<0, \zeta({\frak
%    u}_{y,2},y)- \zeta({\frak u}_{y,1},y)>y-\zeta({\frak u}_{y,1},y)\right]
%\\
%&
%=\bbE\left[\bbP\left[ \bar{T}_{y-z}\left(\zeta(\cdot)\right) >y-z\right]_{z=\zeta({\frak
%  u}_{y,1})}, \, \zeta({\frak u}_{y,1},y)<y\right].\nonumber
%\end{align}
%According to \cite{millar} we have
%$\bbP\left[ \bar{T}_{y-z}\left(\zeta(\cdot)\right) >y-z\right]=1$ for
%all $z\in \bbR$, therefore the right hand side of \eqref{050201-19b}
%equals $1$ and the conclusion of the lemma for $M=1$ has been
%shown. The generalization to an arbitrary $M$ can be achieved via an
%induction argument.
%\qed
%
%\bigskip
%
As a  consequence, we obtain the following estimate on the distance 
the particle travels upon a crossing, so that the jump is "macroscopic". 
% \begin{cor}
% \label{cor010201-19} 
%  Suppose that $M$ is a positive integer. Then, for
% each $x>0$ we have
% \begin{equation}
% \label{070201-19}
% \bbP\left[Z^x({\frak t}_{x,2m-1})>0>Z^x({\frak t}_{x,2m}),\,m=1,\ldots,M\right]=1.
% \end{equation}
% \end{cor} 

% \bigskip

\begin{cor}
\label{cor010301-19} 
 Suppose that $(y,k)\in\bbR_*\times\bbT_*$, $\eps>0$ and $M$ is a
 positive integer. Then,  there exist $C>0$ that depends on $\eps$ and $M$ such that
\begin{equation}
\label{020301-19}
\bbP\left[\min_{m=1,\ldots,M}|\bar\om'(K^k_{n^N_{y,m}})|\bar
  t(K^k_{n^N_{y,m}})\le CN^{\beta/(1+\beta)}\right]<\eps,\quad \hbox{ for all $N\ge 1$.}
\end{equation}
\end{cor} 
{\bf Proof.} Suppose that $y>0$. As a consequence of Theorem~\ref{thm013112}, 
for any $M$, the random vectors
\begin{equation}
\label{040301-19}
(Z_{N}({\frak s}_{y,1}^N,y,k),\ldots, Z_{N}({\frak s}_{y,M}^N,y,k))
\end{equation} converge in law
to  $\left(\zeta({\frak u}_{y,1},y),\ldots, \zeta({\frak u}_{y,M},y)\right)$. 
Lemma \ref{lm010201-19} implies that given $\eps>0$,
there exists~$c>0$ that depends on $\eps$ and $M$ such that
\begin{equation}
\label{070301-19}
\bbP\left[\zeta({\frak u}_{y,2m-1},y)<-c,\, \zeta({\frak
    u}_{y,2m},y)>c,\quad m=1,\ldots,M\right]>1-\eps.
\end{equation} 
%Here $M_1=M/2$ when $M$ is even and $M_1=(M+1)/2$, if otherwise.
Let us set
$$
A_N(c):=\left[Z_{N}({\frak s}_{y,2m-1}^N,y,k)<-c, Z_{N}({\frak
    s}_{y,2m}^N,y,k)>c,\quad m=1,\ldots,M\right].
$$
The convergence in law of the vectors \eqref{040301-19} and (\ref{070301-19})
imply that
\begin{equation}
\label{040301-19a}
\bbP[A_N(c_\eps)]>1-\eps
\end{equation} 
for all sufficiently large  $N$. Decreasing $c>0$ if
necessary, we can claim that \eqref{040301-19a} holds for all $N\ge1$, so that
on  $A_N(c)$ we have  
$$
y-\frac{1}{N^{\beta/(\beta+1)}}\sum_{n=0}^{n_{y,2m-1}^{N}}\bar\om' (K_n^k)\bar t (K_n^k)
<-c<0\le y-\frac{1}{N^{\beta/(\beta+1)}}\sum_{n=0}^{n_{y,2m-1}^{N}-1}
\bar\om' (K_n^k)\bar t (K_n^k)
$$
and
$$
y-\frac{1}{N^{\beta/(\beta+1)}}\sum_{n=0}^{n_{y,2m}^{N}}
\bar\om' (K_n^k)\bar t (K_n^k)>c>0\ge
y-\frac{1}{N^{\beta/(\beta+1)}}\sum_{n=0}^{n_{y,2m}^{N}-1}\bar\om' (K_n^k)\bar
t (K_n^k),
$$
both for all $m=1,\ldots,M$.
Hence, on $A_N(c)$ we hve
$$
|\bar\om' (K_n^k)|\bar t (K_n^k)>cN^{\beta/(\beta+1)},\,m=1,\ldots,M,
$$
% or equivalently 
% \begin{equation}
% \label{020301-19aa}
% \bbP\left[\liminf_{N\to+\infty}\frac{1}{N^{\beta/(\beta+1)}}\min_{m=1,\ldots,M}|\bar\om'\bar t(K^k_{n^N_{y,m}})|=0\right]=0,
% \end{equation}
which in turn yields \eqref{020301-19}.
\qed

\subsection{Processes with reflection, transmission and killing}

\label{sec4.3}

We now restore writing $y,k$ in the notation of the
processes, with $(y,k)\in\bbR\times\bbT_*$. To set the notation for the rescaled
processes, 
let $(\hat\si_{y,m}^{N})$ be a $\{-1,0,1\}$-valued sequence of random
variables that are independent, when conditioned on $(K_n^k)_{n\ge0}$,
and set
$$
\bbP[\hat\si_{y,m}^{N}=0|(K_n^k)_{n\ge0}]= \frak g(K_{n_{y,m}^N}^k),\quad \bbP[\hat\si_{y,m}^{N}=\pm1|(K_n^k)_{n\ge0}]=p_\pm(K^k_{n_{y,m}^{N}}),
$$
as well as 
\begin{equation}
\label{fN}
{\frak f}_N:=\min[m:\,\hat\si_{y,m}^{N}=0],~~
%\end{equation}
% and
%\begin{equation}
%\label{ff}
\tilde{\frak s}_{\frak f}^N:=\tilde{\frak s}^{N}_{y,{\frak
    f}_N}\quad  {\frak s}_{\frak f}^N:={\frak s}^{N}_{y,{\frak f}_N}.
\end{equation}
The killed-reflected-transmitted process $(\tilde Z_N^o(t,y,k))_{t\ge0}$  
can be written as
\begin{equation}
\label{tY}
\tilde Z_N^o(t,y,k):=
 \Big(\prod_{j=1}^{m}\hat\si_{y,j}^{N}\Big)\tilde Z_N(t,y,k),\quad
 t\in[\tilde{\frak s}^N_{y,m}, \tilde{\frak s}^N_{y,m+1}),\,
m=0,1,\ldots.
\end{equation}
We adopt the convention that for $m=0$ the product above equals $1$
and $\tilde{\frak s}^N_{y,0}:=0$.

% Let $(\hat\si_m)_{m\ge1}$ be a  sequence of i.i.d. $\{-1,0,1\}$ random
% variables, independent of $(\zeta(t,y))_{t\ge0}$ 
% with
% $$
% \bbP[\hat\si_m=0]= \frak g_\infty,\quad \bbP[\hat\si_m=\pm1]=p_\pm, \quad m=0,1,\ldots.
% $$
% Let  ${\frak f}:=\min[m\ge 1:\,\si_m=0]$ and
% $
% {\frak t}_f:={\frak t}_{x,{\frak f}}.
% $
% Define the following killed-reflected stable process 
% \begin{equation}
% \label{Zkr}
% \hat \zeta(t,y):=\left(\prod_{j=1}^{m}\hat\si_j\right) Z^x_{t},\quad
% t\in[{\frak s}_{y,m}, {\frak
%                     s}_{y,m+1}),\,
% m=0,1,\ldots.
% \end{equation}
% Here $\tilde{\frak s}_{x,0}:=0$.

For the limit killed-reflected-transmitted 
process, similarly, we let $(\hat\si_m)_{m\ge1}$ be a  sequence of i.i.d. $\{-1,0,1\}$ random
variables, independent of $(\zeta(t,y))_{t\ge0}$, 
with
\begin{equation}
\label{hatsi}
\bbP[\hat\si_m=0]= \frak g_0,\quad \bbP[\hat\si_m=\pm1]=p_\pm, \quad m=0,1,\ldots,
\end{equation}
Here, as we recall,
$
\frak g_0={\frak g}(0)$ and $p_\pm:=p_\pm(0).$
We also set
\begin{equation}
\label{ff}
{\frak f}:=\min[m\ge 1:\,\hat\si_m=0],~~~
{\frak u}_{\frak f}:={\frak u}_{y,{\frak f}},~~{\frak u}_{y,0}:=0.
\end{equation}
The killed-reflected-transmitted stable process has a representation
\begin{equation}
\label{Zkr}
\zeta^o(t,y):=\Big(\prod_{j=1}^{m}\hat\si_j\Big)
\zeta(t,y),\quad t\in[{\frak u}_{y,m}, {\frak
                    u}_{y,m+1}),\,
m=0,1,\ldots.
\end{equation}
%Here ${\frak u}_{y,0}:=0$.

Note that the processes $Z_N(t)$ are discontinuous in $t$ 
while $\tilde Z_N(t)$ are continuous in time. As the process $\zeta(t)$ is
discontinuous, it would not be possible to prove convergence 
of $\tilde Z_N(t)$ to $\zeta(t)$ in the Skorokhod space $D[0,+\infty)$ equipped with the 
$J_1$-topology. Hence, we will need to use the $M_1$-topology that allows convergence
of continuous processes to a discontinuous limit. 
Accordingly, we denote by ${\cal X}$ the space $D[0,+\infty)\times C[0,+\infty)\times
\bar\bbR_+^\bbN$, equipped with the product  of $M_1$
and uniform convergence on compacts topologies in the first two
variables and the standard product topology on $\bar\bbR_+^\bbN$.
We will use below the metric~${\rm d}_\infty $ that metrizes the~$M_1$-topology
on $D[0,+\infty)$, see 
Appendix~\ref{sec:appB1} for a brief review of the required definitions. 
Our main result in this section is the following.
\begin{thm}
\label{thm010401-19}
The random elements  
%\begin{equation}
%\label{051102-19a}
$\Big((\tilde
    Z_N^o(t,y,k))_{t\ge0},(\fT_N(t,k))_{t\ge0},(\tilde{\frak
      s}^N_{y,m})_{m\ge1}\Big)$
%\end{equation}
 converge in law over
${\cal X}$ to 
the random element  
%\begin{equation}
%\label{051102-19}
$\Big((\zeta^o(t,y))_{t\ge0},(\tau(t))_{t\ge0}, ({\frak
      u}_{y,m})_{m\ge1}\Big).
$
%\end{equation}
\end{thm}
{\bf Proof.}
Let us define the process
\begin{equation}
\label{tY1}
Z_{N}^o(t):=
\Big(\prod_{j=1}^{m}\hat\si_{y,j}^{N}\Big) Z_{N}(t,y,k),\quad t\in[\tilde{\frak s}^N_{y,m}, \tilde{\frak s}^N_{y,m+1}),\,
m=0,1,\ldots.
\end{equation}
It is straightforward to show that for any $\eta>0$ we have
\begin{align}
\label{042102ab}
\lim_{N\to+\infty}\bbP\left[\mbox{d}_\infty( Z_{N}^o, \tilde Z_N^o)\ge
  \eta\right]=0.
\end{align}
Therefore, we may now pass from $\tilde Z_N^o$ to $Z_N^o$ and prove convergence
of the discontinuous processes $Z_N^o$ to the discontinuous jump process. This 
can be done using the $J_1$-topology as both processes are discontinuous,
and is simpler than working directly
in the $M_1$-topology. 
Accordingly, ${\cal X}'$ be the space ${\cal X}$, equipped with the product topology,
where on the first component we put the $J_1$ topology rather than $M_1$.
We will prove that 
the random elements 
\[\Big((Z_N^o(t,y,k))_{t\ge0},(\fT_N(t,k))_{t\ge0},({\frak s}^N_{y,m})_{m\ge1}\Big)
\]
converge in law over ${\cal X}'$ to 
$\Big((\zeta^o(t,y))_{t\ge0},(\tau(t))_{t\ge0}, ({\frak
      u}_{y,m})_{m\ge1}\Big).$
Thanks to \eqref{010301-19} and \eqref{042102ab} this will
finish the proof of the theorem. Since we have already proved
the convergence of the last two components, see Proposition
\ref{prop010801-19} and Theorem \ref{thm013112}, we focus only on
proving the convergence in law of
$(Z_{N}^o(t))_{t\ge0}$ over $D[0,+\infty)$, equipped with the $J_1$-topology to
$(\zeta^o(t,y))_{t\ge0}$.

\begin{lemma}
\label{lm020401-19} For any $\eps>0$ there exists $M_\eps>0$ such that 
%\begin{equation}
%\label{050401-19}
$\bbP\left[{\frak f}_N\ge M_\eps\right]<\eps$ for all~$N\ge 1$.
%\quad N\ge 1.
%\end{equation}
\end{lemma}
{\bf Proof.}
From the continuity of ${\frak g}(\cdot)$ and its strict positivity we have
$$
\delta:=\inf_{k\in\bbT}{\frak g}(k)>0.
$$
The definition of the sequence $(\hat\si_{y,m}^{N})$ implies that
$$
 \bbP\left[{\frak f}_N\ge M\right]=\bbP\left[\hat\si_{y,1}^{N},\ldots
   ,\hat\si_{y,M}^{N}\in\{-1,1\}\right]\le (1-\delta)^M
$$
and the conclusion of the lemma follows, upon a choice of a
sufficiently large $M$.
\qed

\bigskip

Let $(U_m)_{m\ge1}$ be a sequence of i.i.d. random 
variables, uniformly distributed in~$(0,1)$, independent of the
sequence $(K_n^k)_{n\ge0}$. Let us define
\begin{equation}
\label{sim}
\hat\si_{y,m}^{N}=1_{(0, \frak p_{m,+}^N)}(U_m)-1_{(1-\frak p_{m,-}^N,1)}(U_m)
\end{equation}
and
\begin{equation}
\label{si}
\hat\si_m:=1_{(0, p_{+})}(U_m)-1_{(1-p_{-},1)}(U_m),\quad m\ge1,
\end{equation}
where
%$ \frak g_m^N:= \frak g(K^k_{n_{y,m}^{N}})$ and~
$ \frak p_{m,\pm}^N:=p_\pm(K^k_{n_{y,m}^{N}})$. 
\begin{lemma}
\label{lm010401-19}
For any integer $M>0$ and $\eps>0$ we have
\begin{equation}
\label{030401-19}
\bbP\Big[\bigcup_{m=1}^M[\hat\si_{y,m}^{N}\not=\hat\si_m]\Big]< 
\Big(\frac{c_+\om_0'}{C\ga_0 R_0}\Big)^{\ga/\beta}\frac{2  C_0M}{N^{\ga/(\beta+1)}}+\eps,\quad N\ge 1,
\end{equation}
where $C>0$ is as in Corollary
\ref{cor010301-19}, while $C_0,\ga>0$ are as in \eqref{non-deg1}.
\end{lemma}
{\bf Proof.}
Consider the event
$$
A_N:=\Big[\min_{m=1,\ldots,M}|\bar\om'\bar t (K^k_{n_{y,m}^{N}})|\le
  CN^{\beta/(\beta+1)}\Big],
$$
where  $C$ is as in the statement of Corollary
\ref{cor010301-19}, and write
\begin{align}
\label{070301-19n}
&\bbP\Big[\bigcup_{m=1}^M\left[\hat\si_{y,m}^{N}\not=\hat\si_m\right]\Big]\le
 \bbP\Big[\bigcup_{m=1}^M\left[\hat\si_{y,m}^{N}\not=\hat\si_m\right],A_N^c\Big]+\bbP[A_N]
\nonumber  \\
&
\le
 \bbP\Big[\bigcup_{m=1}^M[\hat\si_{y,m}^{N}\not=\hat\si_m],\, 
 \min_{m=1,\ldots,M}|\bar\om'\bar t (K^k_{n_{y,m}^{N}})|>
  CN^{\beta/(\beta+1)}\Big]+\eps\\
&
\le
\sum_{m=1}^M
 \bbP\left[\hat\si_{y,m}^{N}\not=\hat\si_m,\, |\bar\om'\bar t (K^k_{n_{y,m}^{N}})|>
  CN^{\beta/(\beta+1)}\right]+\eps. \nonumber  
\end{align}
Note that, for all $m\ge 1$ we have
 \begin{align*}
&
\bbP\left[\hat\si_{y,m}^{N}\not=\hat\si_m,\, |\bar\om'\bar t (K^k_{n_{y,m}^{N}})|>
  CN^{\beta/(\beta+1)}\right]\\
&
\le \bbE\Big[|1_{(0, \frak p_{m,+}^N)}(U_m)\!-\!1_{(0, p_{+})}(U_m)|\!+\!|1_{(1-\frak p_{m,-}^N,1)}(U_m) \!-\!
1_{(1-p_{-},1)}(U_m)|,\,  |\bar\om'\bar t (K^k_{n_{y,m}^{N}})|\!>\!
  CN^{\beta/(\beta+1) }\Big]\\
&
=\sum_{\iota=\pm}\bbE\left[| p_\iota(K^k_{n_{y,m}^{N}})- p_{\iota}|,\, |\bar\om'\bar t (K^k_{n_{y,m}^{N}})|>
  CN^{\beta/(\beta+1)}\right].
\end{align*}
By virtue of \eqref{non-deg1} and \eqref{cp}, the right side can be
estimated by
\begin{align*}
\sum_{\iota=\pm}\bbE\Big[\left| p_\iota(K^k_{n_{y,m}^{N}})-
  p_{\iota}\right|,\,
  |K^k_{n_{y,m}^{N}}|<\Big(\frac{c_+\om_0'}{C\ga_0 R_0}\Big)^{1/\beta}\frac{1}{N^{1/(1+\beta)}}
 \Big]
\le2 C_0 \Big(\frac{c_+\om_0'}{C\ga_0 R_0}\Big)^{\ga/\beta}\frac{1}{N^{\ga/(\beta+1)}},%\quad m=1,\ldots,M
\end{align*}
and \eqref{030401-19} follows.
\qed

Next, we define the process
\begin{equation}
\label{tY1p}
\hat Z_{N}^o(t):=
 \Big(\prod_{j=1}^{m}\hat\si_{j}\Big) Z_{N}(t,y,k),\quad t\in[\tilde{\frak s}^N_{y,m}, \tilde{\frak s}^N_{y,m+1}),\,
m=0,1,\ldots,
\end{equation}
with the random variables $\si_m$ given by \eqref{si}.
Using Lemma \ref{lm020401-19} to choose $M$ large enough,
and then Lemma~\ref{lm010401-19} to choose $N$ large, we conclude the following.
\begin{cor}
\label{cor010401-19}
Let $\left(Z_{N}^o(t)\right)_{t\ge0}$ be %the process
defined by  \eqref{tY1} with  
$\left(\hat\si_{y,m}^N\right)_{m\ge 1}$ given 
by~\eqref{sim}. Then,  for any $\eps>0$ there exists $N_0$ such that
\begin{equation}
\label{040401-19}
\bbP\left[ Z_{N}^o\not=\hat Z_{N}^o\right]<\eps,\quad N\ge N_0.
\end{equation}
\end{cor}
Theorem \ref{thm013112} and Corollary \ref{cor010401-19} immediately
imply Theorem \ref{thm010401-19}.
\qed

Given any $t_0\ge0$ the limiting process $\left(\zeta^o(t,y)
\right)_{t\ge0}$ is a.s. continuous at $t_0$, as a consequence of
an analogous property of $\left( \zeta(t,y)
\right)_{t\ge0}$ mentioned earlier  (see  Proposition
1.2.7 p. 21 of \cite{bertoin}). It follows that the coordinate
mapping is  continuous on an event of probability one in the $M_1$
topology, see Theorem 12.5.1 part (v) of \cite{whitt}. As a
consequence we conclude the following.
\begin{cor}
\label{cor010902-19}
The processes  $(\tilde Z_N^o(t,y,k))_{t\ge0}$ converge in the sense of
finite-dimensional distributions, as $N\to+\infty$,  to
the process $\left(\zeta^o(t,y)\right)_{t\ge0}.$
\end{cor}

\subsection{The re-scaled process for the kinetic equation}

Let us now introduce the process corresponding to the kinetic equation
(\ref{kinetic-sc0}) with reflection-transmission-killing at the interface:
\begin{equation}
\label{030902-19}
 Y_N^o(t):=\tilde Z_N^o(\fT_N^{-1}(t;k)),\quad t\ge0,
\end{equation}
where $\fT_N$ and $\tilde Z_N^o$ are given by \eqref{010502-19N} and
\eqref{tY}, respectively.
We set
\begin{equation}
\label{tsn}
\hat {\frak s}_{y,m}^N:=\fT_N(\tilde {\frak
  s}_{y,m}^N;k)\quad\mbox{and}\quad \hat {\frak s}_{\frak
  f}^N:=\fT_N(\tilde {\frak s}_{y,{\frak f}_N}^N;k),
\end{equation}
where ${\frak f}_N$ is as in \eqref{fN}.
We let furthermore $\tilde \si^N(t)\equiv1$ for 
 $t\in [0, \tilde {\frak s}_{y,1}^N)$ and 
\begin{equation}
\label{050507yt}
\tilde \si_N(t):=
\prod_{j=1}^{m}\hat\si_{y,j}^{N},\quad t\in [\tilde {\frak
  s}_{y,m}^N,\tilde  {\frak s}_{y,m+1}^N),\, m\ge1,
\end{equation}
%where 
%$\tilde\si_{y,j}^{N}=1$, if $\hat\si_{y,j}^{N}=0,1$ and
%$\tilde\si_{y,j}^{N}=-1$, if $\hat\si_{y,j}^{N}=-1$. 
and
\begin{equation}
\label{050507x}
 K_N^o(t,y):=\si_N^o(t) K_N(t,k),
\end{equation}
with $ \si_N^o(t) :=\tilde \si_N(\fT_N^{-1}(t))$ and 
\begin{equation}
\label{KNt}
K_N(t,k):=K_{[\fT_N^{-1}(t;k)]}^k.
\end{equation}
As in \eqref{040902-19},
we have
\begin{align}
\label{040902-19a}
\frac{d Y^o_N(t;y,k)}{dt}=-\om'( K^o_N(t;k)),\quad t\in[0, \hat {\frak s}_{\frak f}^N).
\end{align}
We also set $\hat{\frak u}_{y,m}:=\tau^{-1}({\frak u}_{y,m})$ and
\begin{equation}
\label{050902}
\eta^o(t,y):=\zeta^o(\tau^{-1}(t),y), 
\end{equation}
with $\zeta^o$ and $\tau$ given by \eqref{Zkr} and~\eqref{tau}, respectively.
The following is a direct corollary of Theorems \ref{thm013112} and
\ref{thm010401-19}.
\begin{thm}
\label{thm010902-19}
The random elements 
%\begin{equation}
%\label{071102-19}
$((Y_N^o(t))_{t\ge0},(\hat
    {\frak s}_{y,m}^N)_{m\ge1})$
%\end{equation} 
converge in law to 
\[
((\eta^o(t,y))_{t\ge0},
(\hat{\frak u}_{y,m})_{m\ge1}),
\]
over $D[0,+\infty)\times \bar\bbR_+^\bbN$,
with the product of the $M_1$ and standard product topologies.
\end{thm} 
{\bf Proof.} By Theorem \ref{thm010401-19}  and the Skorochod embedding
theorem we can find equivalent versions of %the random elements
$((\tilde Z_N^o(t,y,k))_{t\ge0},(\fT_N(t,k))_{t\ge0})$ 
converging a.s. to $((\zeta^o(t,y))_{t\ge0},(\tau(t))_{t\ge0})$. 
Hence,~$\fT_N^{-1}$ converge a.s. in the uniform topology on compacts to $\tau^{-1}$.
Invoking Theorem 7.2.3 p. 164 of \cite{whitt-s}, we conclude
convergence of  the $(Y_N^o(t))_{t\ge 0}$ to $((\eta^o(t,y))_{t\ge0}$.
%first components of the random elements \eqref{071102-19}. 
The
convergence of the second components is a consequence of \eqref{tsn}.
\qed

\section{The proof of convergence in  Theorem \ref{main-thm} }
\label{sec-proof}

It suffices to prove the convergence statement for 
\begin{equation}
\label{022205-19}
\widetilde W_N(t,y,k):=W_N(t,y,k)-T.
\end{equation}
It satisfies \eqref{kinetic-sc0} with the initial condition $\widetilde W_0
:=W_0-T\in {\cal C}_0$, so that 
the interface conditions~\eqref{feb1408},~\eqref{feb1410}
correspond to $T=0$.
% This ends the proof of
% \eqref{081102-19} concluding in this way the proof of 
% the convergence
We will show that 
\begin{equation}
\label{010701-19NN}
\lim_{N\to+\infty}\int_{\bbR\times\bbT}\widetilde
W_N(t,y,k)G(y,k)dydk=\int_{\bbR\times\bbT}\widetilde W(t,y)G(y,k)dydk,
\end{equation}
for any $G\in C_c^\infty(\bbR\times\bbT)$, where
\begin{equation}
\label{010701-19NN1}
\widetilde W(t,y)=\bbE\left[\bar W_0\left(\eta^o(
      t;y)\right),\, t<\hat  {\frak u}_{y,{\frak
  f}}\right]
\end{equation}
and
\begin{equation}
\label{barW}
\bar W_0(y):=\int_{\bbT}\widetilde W_0(y,k)dk.
\end{equation}
% We claim that $\widetilde W(t,y)$ satisfies 
% \begin{align}
% \label{013101-19-1}
% &
% \partial_t\widetilde W(t,y)=\hat c \partial_{y}^{1+1/\beta} \widetilde W(t,y)+\hat c p_-\int_{[yy'<0]}q_\beta(y-y')[\widetilde W(t,-y')-\widetilde W(t,y')]dy'\nonumber\\
% &
% -\hat c\frak g_0\int_{[yy'<0]}q_\beta(y-y')\widetilde W(t,y')dy',\\
% &
% \widetilde W(0,y):=\int_{\bbT}\widetilde W_0(y,k)dk.\nonumber
% \end{align}
This will imply that
\begin{align}
\label{010701-19NN1a}
&
\lim_{N\to+\infty}\int_{\bbR\times\bbT}
W_N(t,y,k)G(y,k)dydk=T\int_{\bbR\times\bbT}G(y,k)dydk\nonumber\\
&
+\lim_{N\to+\infty}\int_{\bbR\times\bbT}\widetilde
 W_N(t,y,k)G(y,k)dydk =T\int_{\bbR\times\bbT}G(y,k)dydk\\
&
+\lim_{N\to+\infty}\int_{\bbR\times\bbT}\widetilde
 W(t,y)G(y,k)dydk=
\int_{\bbR\times\bbT}W(t,y)G(y,k)dydk,\nonumber
\end{align}
where 
\begin{equation}
\label{012603-19}
 W(t,y)=T+\widetilde W(t,y).
\end{equation}
% Note that the constant function $T$ satisfies \eqref{013101-19}. Thus
% $ W(t,y)$ satisfies  \eqref{013101-19} with the initial data 
% $$
%  W(0,y)=T+\widetilde W(0,y)=\int_{\bbT} W_0(y,k)dk.
% $$

%\subsubsection*{Proof of $(\ref{010701-19NN})$}

We now prove (\ref{010701-19NN}).
Using Proposition  \ref{prop010701-19a}, we   write  
\begin{align}
\label{010701-19N}
&
\widetilde W_N(t,y,k)=\bbE\left[\widetilde W_0\left(Y^o_N( t;y,k), K^o_N( t,k)\right),\, t\le \tilde  {\frak s}^N_{y,{\frak f}_N}\right].
\end{align}
For a given test function $G\in C_c^\infty(\bbR\times \bbT)$ let us set
\begin{equation}
\label{061102-19}
I_N:=\int_{\bbR\times\bbT}W_N(t,y,k)G(t,y,k)dydk.
\end{equation}
% where the terms in the right hand side correspond to the term in the
% right hand side of \eqref{010701-19N}. Using Lemmas \ref{lm020401-19}
% and \ref{lm010401-19}, together with the conclusion of Theorem
% \ref{thm010902-19} we can argue that (cf \eqref{ff})
% \begin{equation}
% \label{071102-19}
% \lim_{N\to+\infty}I\!I_N=T{\frak g}_0\int_{\bbR\times\bbT}\bbP[t>\hat u_{y,{\frak f}}]G(y,k)dydk,
% \end{equation}
Our goal is to show that
\begin{equation}
\label{081102-19}
\limsup_{N\to+\infty}\left|I_N-\int_{\bbR\times\bbT}\bbE\left[\bar W_0\left(\eta^o( t;y)\right),\, t<\hat  {\frak u}_{y,{\frak f}}\right]G(y,k)dydk\right|<\eps,
\end{equation}
where $\eps>0$ is arbitrary and $\bar W_0(y)$ is given by \eqref{barW}.
% \begin{equation}
% \label{barW}
% \bar W_0(y):=\int_{\bbT}\widetilde W_0(y,k)dk.
% \end{equation}
Since $\widetilde W_0$ is continuous outside %possibly  
the interface $[y=0]$, for any $\delta>0$ we
can write that
$$
\widetilde W_0(y,k)=W_0^1(y,k)+W_0^2(y,k),
$$
where $W_0^2\in C_b(\bbR\times\bbT)$, and
\begin{equation}
\label{012604-19}
{\rm supp}\,W_0^1\subset [|y|<2\delta]\times\bbT,\quad {\rm supp}\,W_0^2\subset [|y|>\delta/2]\times\bbT,\hbox{ 
and $\|W_0^j\|_\infty\le \|\widetilde W_0\|_\infty$, $j=1,2$.}
\end{equation}
We can decompose
accordingly
$I_N=I_N^1+ I_N^2$ and $\bar W_0=\bar W_0^1 +\bar W_0^2$.
Then, we have
\begin{align}
\label{091102-19a}
&
\limsup_{N\to+\infty}\left|I_N^1-\int_{\bbR\times\bbT}\bbE\left[\bar W_0^1\left(\eta^o(
      t;y)\right),\, t<\hat  {\frak u}_{y,{\frak
  f}}\right]G(y,k)dydk\right| \\
&
\le
\|\widetilde W_0\|_\infty \int_{\bbR\times\bbT}|G(y,k)|
\left(\limsup_{N\to+\infty}\bbP\left[|Y_N( t;y,k)|<2\delta\right]
+\bbP\left[|\eta( t;y)|<2\delta\right]\right) dydk \nonumber
<\frac{\eps}{10},
\end{align}
provided that $\delta>0$ is sufficiently small.

Let us set
\begin{equation}
\label{I2}
I_N^2(\delta):=
\int_{\bbR\times\bbT}\bbE\left[W_0^2\left(Y^o_N( t-\delta;y,k), K^o_N(
    t,k)\right),\, t-\delta\le \tilde  {\frak s}^N_{y,{\frak
      f}_N}\right] G(y,k)dydk .
\end{equation}
By virtue of Lemma \ref{lm020401-19} and Theorem  \ref{thm010902-19}
 we can  write
\begin{align}
\label{081102-19a}
&
\limsup_{N\to+\infty}\left|I_N^2-I_N^2(\delta)\right|\nonumber\\
&
\le
\limsup_{N\to+\infty}\int_{\bbR\times\bbT}\bbE\left[\sup_{k'}\left|W_0^2\left(Y^o_N( t;y,k), k'\right) -W_0^2\left(Y^o_N( t-\delta;y,k), k'\right)\right|\right]|G(y,k)|dydk\nonumber
\\
&
+\|\widetilde W_0\|_\infty\limsup_{N\to+\infty}\int_{\bbR\times\bbT}\bbP\left[t-\delta\le \tilde  {\frak s}^N_{y,{\frak f}_N}<t\right] |G(y,k)|dydk<\frac{\eps}{10}
\end{align}
if $\delta>0$ is sufficiently small. We have used here the fact that,
for each $m\ge1$  the
law of~$(\hat  {\frak u}_{y,1},\ldots, \hat  {\frak u}_{y,m})$
-- the limit of the laws of  $(\tilde  {\frak s}^N_{y,1},\ldots,
\tilde  {\frak s}^N_{y,m})$, as $N\to+\infty$, -- is
absolutely continuous with respect to the Lebesgue measure. 
This is a
consequence of the strong Markov property of
$(\eta(t,y))_{t\ge0}$ and the fact that the joint law of
$(\hat  {\frak u}_{y,1},\eta(\hat  {\frak u}_{y,1},y))$ is
absolutely continuous with respect to the Lebesgue measure, see e.g.
Theorem 1, p. 93 of~\cite{ikeda-watanabe}.

To conclude
\eqref{081102-19}, it suffices to prove that we can choose  a sufficiently small
$\delta>0$ so that
\begin{align}
\label{081102-19b}
&
\limsup_{N\to+\infty}\left|I_N^2(\delta)-\int_{\bbR\times\bbT}\bbE\left[\bar W_0^2\left(\eta^o(t;y)\right),\, t<\hat  {\frak u}_{y,{\frak
  f}}\right]G(y,k)dydk\right|<\frac{\eps}{10}.
\end{align}
To this end, we will assume, without loss of any generality, that $W_0^2\in
C^\infty_c(\bbR\times \bbT)$. Indeed, for any~$W_0^2\in
C_b(\bbR\times \bbT)$ satisfying~\eqref{012604-19} and $R>0$, we can find  $W_0^{2,s}\in
C^\infty_c(\bbR\times \bbT)$ and    such that
$$
\hbox{$\| W_0^{2,s}\|_\infty\le \|
W_0^{2}\|_\infty+1$ and } 
\sup_{|y|\le R,k\in \bbT}|W_0^{2,s}(y,k)-W_0^{2}(y,k)|<\frac{\eps}{100}.
$$
Thanks to the already established  tightness of the laws of
$Y_N^o(t,y,k)$ we can easily see that, upon the choice of a
sufficiently large $R>0$,
$$
\limsup_{N\to+\infty}\left|I_N^2(\delta)-I_N^{2,s}(\delta)\right|<\frac{\eps}{10},
$$
where $I_N^{2,s}(\delta)$ is defined by  
\eqref{I2}, with   $W_0^{2,s}$ replacing $W_0^{2}$.  From here
on, we will restrict our attention to $I_N^{2,s}(\delta)$.

Using Lemmas \ref{lm020401-19}
and \ref{lm010401-19}, together with the conclusion of Theorem
\ref{thm010902-19} one can show that for a sufficiently small $\delta>0$ we have
\begin{align}
\label{081102-19c}
\limsup_{N\to+\infty}\left|I_N^2(\delta)-\tilde I_N^2(\delta)\right|<\frac{\eps}{10},
\end{align}
where 
$$
\tilde I_N^2(\delta):=
\int_{\bbR\times\bbT}\bbE\left[W_0^2\left(Y^o_N( t-\delta;y,k),\hat K^o_N(
    t,k)\right),\, t-\delta\le \tilde  {\frak s}^N_{y,{\frak
      f}_N}\right] G(y,k)dydk ,
$$
and
\begin{equation}
\label{050507xh}
 \hat K_N^o(t,y):=\hat\si_N^o(t-\delta) K_N(t,k),
\end{equation} where
    $\hat \si_N^o(t):=\hat \si_N(T_N^{-1}(t))$,
$\hat \si_N(t)\equiv1$ for 
 $t\in [0, \tilde {\frak s}_{y,1}^N)$, and 
\begin{equation}
\label{050507y}
\hat \si_N(t):=
\prod_{j=1}^{m}\hat\si_{j},\quad t\in [\tilde {\frak
  s}_{y,m}^N,\tilde  {\frak s}_{y,m+1}^N),\, m\ge1.
\end{equation}
 Conditioning on ${\cal K}_{t-\delta}$, where
$\left({\cal K}_t^N\right)_{t\ge0}$ is the natural filtration of
$(K_N(t,k))_{t\ge0}$, we write
\begin{align*}
\tilde I_N^2(\delta)=
\hat I_N^2(\delta)+\bar I_N^2(\delta),
\end{align*}
where
\begin{align*}
&\hat I_N^2(\delta):=
\int_{\bbR\times\bbT}\bbE\left[\bar W^2_0(Y^o_N( t-\delta;y,k)),\, t-\delta\le \tilde  {\frak s}^N_{y,{\frak
      f}_N}\right] G(y,k)dydk ,\\
&
\bar I_N^2(\delta):=\int_{\bbR\times\bbT}\bbE\left[w_N(Y^o_N( t-\delta;y,k), \hat\si_N^o(t-\delta) K_N(t-\delta,k)),\, t-\delta\le \tilde  {\frak s}^N_{y,{\frak
      f}_N}\right] G(y,k)dydk
\end{align*}
and
$$
\bar W^2_0(y):=\int_\bbT W^2_0(y,k)\mu(dk).
$$
We have used above the notation
$$
w_N(y,k):=e^{N\delta
  L}W_0^2\left(y,\cdot\right)(k)=\frac{1}{2\pi}\sum_{\ell\in\bbZ\setminus\{0\}}e^{N\delta
  L}{\frak e}_\ell(k)\int_{\bbR}\widehat
W_0^2\left(\xi,\ell\right)e^{i\xi y}d\xi,
$$
with the generator $L$ given by \eqref{L},
${\frak e}_\ell(k):=\exp\big\{2\pi i k\ell\big\}$ and 
$$
\widehat
W_0^2\left(\xi,\ell\right)=\int_{\bbR\times \bbT}
W_0^2\left(y,k\right)e^{-i\xi y}{\frak e}_\ell^\star(k)dy dk.
$$
The term $\bar I_N^2(\delta)$ we can estimated as follows:
\begin{align}
\label{081102-19d}
&\left|\bar I_N^2(\delta)\right|\le
  \int_{\bbR}\sup_{k\in\bbT}|G(y,k)|dy\\
&
\times \sum_{\ell\in\bbZ\setminus\{0\}}\int_{\bbR\times \bbT}|\widehat
W_0^2\left(\xi,\ell\right)|\bbE\Big[|e^{N\delta
  L}{\frak e}_\ell(K_N(t-\delta,k))|+|e^{N\delta
  L}{\frak e}_\ell(-K_N(t-\delta,k))|\Big]d\xi dk\nonumber\\
&
= 2\int_{\bbR}\sup_{k\in\bbT}|G(y,k)|dy \sum_{\ell\in\bbZ\setminus\{0\}}\|e^{N\delta
  L}{\frak e}_\ell\|_{L^1(\bbT)}\int_{\bbR}|\widehat
W_0^2\left(\xi,\ell\right)|d\xi \nonumber.
\end{align}
Now, \eqref{etL} implies that $\|e^{N\delta
  L}{\frak e}_\ell\|_{L^1(\bbT)}\to0$, as $N\to+\infty$, for each $\ell\not=0$. Therefore,
$$
\lim_{N\to+\infty}\bar I_N^2(\delta)=0.
$$
It follows from Theorem \ref{thm010902-19}   that
\begin{align}
\label{081102-19b-bis}
&
\limsup_{N\to+\infty}\left|\hat I_N^2(\delta)-\int_{\bbR\times\bbT}\bbE\left[\bar W_0^2\left(\eta^o(t;y)\right),\, t<\hat  {\frak u}_{y,{\frak
  f}}\right]G(y,k)dydk\right|<\frac{\eps}{2},
\end{align}
provided that $\delta>0$ is sufficiently small. 
This ends the proof of \eqref{010701-19NN}.

\section{Proof of Theorem \ref{main-thm}: description of the limit}
\label{sec8}

So far, we have shown the weak convergence of $W_N(t,y,k)$ to $ W(t,y)$,
in the sense of~(\ref{010701-19NN}), with $W(t,y)$ defined in
\eqref{012603-19} and  (\ref{010701-19NN1}). 
We now identify $ W(t,y)$ as a weak solution to (\ref{013101-19x}) if 
$W_0\in {\cal C}_T$.
Thanks to \eqref{022205-19} and \eqref{012603-19} it suffices only to
consider the case $T=0$. Consider a regularized scattering kernel: take $a\in(0,1)$ and set
$$
q^{(a)}_\beta(y)=\frac{c_\beta1_{(a,+\infty)}(|y|)}{|y|^{2+1/\beta}},\quad y\in\bbR_*.
$$
Let $\left(\eta_a(t,y)\right)_{t\ge0}$ be a Levy
process starting at $y\in\bbR$, with the generator $-\hat c\Lambda_{\beta}^{(a)}$, where 
\begin{equation}
\label{La}
\Lambda_{\beta}^{(a)}F(y):=\int_{\bbR}q^{(a)}_\beta(y-y')[F(y)-F(y')]dy',\quad
F\in B_b(\bbR).
\end{equation}
It is well known, see e.g. Section 2.5 of \cite{kyprianou}, that
$\left(\eta_a(t,y)\right)_{t\ge0}$
converge in law, as $a\to 0^+$, over~$D[0,+\infty)$, with the topology of the  uniform
convergence on compacts, to $\left(\eta(t,y)\right)_{t\ge0}$,
the symmetric stable process with the generator $-\hat c\Lambda_{\beta}$,
as in  \eqref{frac-lap}.

We define inductively the times of jumps over the interface
$$
\hat {\frak u}^a_{y,1}:=\inf[t>0:\, \eta_a(t-,y)\eta_a(t,y)<0],\quad \hat {\frak u}^a_{y,m+1}:=\inf[t> \hat {\frak u}^a_{y,m}:\, \eta_a(t-,y)\eta_a(t,y)<0].
$$
To set up the reflected-transmitted-killed process,
let $(\hat\si_m)_{m\ge1}$ be a  sequence of i.i.d. $\{-1,0,1\}$ random
variables, independent of $\left(\eta_a(t,y)\right)_{t\ge0}$
distributed according to \eqref{hatsi}, and set 
\begin{equation}
\label{eta0}
\eta_a^o(t,y):=\Big(\prod_{j=1}^m\hat\si_j\Big)\eta_a(t,y),\quad
t\in[\hat {\frak u}^a_{y,m}, \hat {\frak u}^a_{y,m+1}),\,m\ge0,
\end{equation}
where $\hat {\frak u}^a_{y,0}:=0$.
Using Theorem \ref{thmA} together with the argument in Section
\ref{sec4.3} we easily conclude   the following.
\begin{thm}
\label{thm011402-19}
The random elements $((\eta_a^o(t,y))_{t\ge0},(\hat {\frak u}^a_{y,m})_{m\ge1})$ 
converge in law over the product space
$D[0,+\infty)\times\bar\bbR_+^\bbN$, equipped with the product of the
topology of uniform convergence on compacts and the standard infinite
product topology, to
$((\eta^o(t,y))_{t\ge0},(\hat {\frak
      u}_{y,m})_{m\ge1})$, as $a\to0$.
\end{thm}

As a direct corollary of the above theorem we conclude that
\begin{equation}
\label{011402-19}
\lim_{a\to0+}\widetilde W^{(a)}(t,y)
=\widetilde W(t,y),\quad (t,y)\in \bar\bbR_+\times\bbR_*,
\end{equation}
with $\widetilde W(t,y)$, the limit of $W_N(t,y,k)$, given by  \eqref{010701-19NN1}, and
\begin{equation}
\label{010701-19NN1a-bis}
\widetilde W^{(a)}(t,y)=\bbE\left[\bar W_0\left(\eta^o_a( t;y)\right),\, t<\hat  {\frak u}^a_{y,{\frak
  f}}\right],
\end{equation}
where $\bar W_0$ is given by \eqref{barW}, and ${\frak f}$  by \eqref{ff}. 

Note that $\widetilde W^{(a)}(t,y)$ satisfies
\begin{align}
\label{013101-19a}
\partial_t\widetilde W^{(a)}(t,y)=\hat c \hat L_a\widetilde
  W^{(a)}(t,y),\quad (t,y)\in\bbR_+\times\bbR_*,
\end{align}
in the classical sense, where
\begin{align*}
\hat L_a F(y):=&-\Lambda_\beta^{(a)} F(y)+p_- \int_{[yy'<0]}q_\beta(y-y')[F(-y')-F(y')]dy'
\\
&
-{\frak g}_0\int_{[yy'<0]}q_\beta(y-y')F(y')dy',\quad F\in B_b(\bbR).
\end{align*} 
Indeed, let
$$
Q^{(a)}:=\hat c\int_{\bbR}q_\beta^{(a)}(y)dy,
$$
and for $\Delta t\ll1$ and $t>0$  write  
\begin{align*}
\widetilde W^{(a)}(t+\Delta t,y)&=\bbE [\bar W_0(\eta_a^o(t+\Delta t,y)) ,\, t+\Delta t<\hat  {\frak u}^a_{y,{\frak
  f}}]
=\bbE \left[\bar W_0(\eta_a^o(t, \eta_a^o(\Delta t,y)),  t<\hat  {\frak u}^a_{ \eta_a^o(\Delta t,y),{\frak
  f}}\right]\\
&
=e^{-Q^{(a)}\Delta t}\widetilde W^{(a)}(t,y)+\hat c%\left\{
\int_{[yy'>0]}q_\beta(y-y')\widetilde W^{(a)}(t,y')dy'%\right\}
\Delta t\\
&
+\mathop{\underbrace{(1-p_--{\frak g}_0)}}_{=p_+}\hat
  c%\left\{
  \int_{[yy'<0]}q_\beta(y-y')\widetilde W^{(a)}(t,y')dy'%\right\}
  \Delta
  t\\
&
+p_-\hat c%\left\{
\int_{[yy'>0]}q_\beta(y+y')\widetilde W^{(a)}(t,y')dy'%\right\}
\Delta t+o(\Delta t).
\end{align*}
It follows that
$$
\widetilde W^{(a)}(t+\Delta t,y)-\widetilde W^{(a)}(t,y)=\hat c\hat L_a \widetilde W^{(a)}(t,y)\Delta t+o(\Delta t),
$$
which implies \eqref{013101-19a}.
Thanks to \eqref{011402-19}, we 
conclude that $\widetilde W$ satisfies % part (ii) of
Definition
\ref{df011803-19}.

\section{Proof of Theorem \ref{thm012205-19}}
\label{sec:proof-theor-refthm01}

{By considering the kinetic equation with the initial data $W_0':=W_0-T'$
we may assume that $T'=0$ and
$W_0\in L^1(\bbR\times\bbT)$.}

Assume first that $T=0$. Let %the Hilbert space norm
$\|\cdot\|_{{\cal H}_a}$ be defined by the analog to \eqref{Hnorm},
with the kernel $q^{(a)}_\beta(\cdot)$ replacing  $q_\beta(\cdot)$, and
the Hilbert space ${\cal H}_a$ be the completion of
$C_0^\infty(\bbR_*)$ in the respective norm.
Obviously, we have
\begin{equation}
\label{012003-19}
\|G\|_{{\cal H}_a}\le \|G\|_{{\cal H}_{a'}},\quad G\in
C_0^\infty(\bbR_*),\,a>a'\ge 0.
\end{equation}
As with (\ref{apr210}), we have
%The following result follows upon a direct calculation, multiplying
%both sides of \eqref{013101-19a} by $\widetilde W^{(a)}(t,y)$ and
%integrating in the $y$ variable, cf \eqref{apr210}.
%\begin{prop}
%\label{prop011903-19}
%We have
\begin{align}
\label{011903-19}
&\frac{d}{dt}\|\widetilde W^{(a)}(t)\|^2_{L^2(\bbR)}=-\hat
  c \|\widetilde W^{(a)}(t)\|_{{\cal H}_a}^2,
\end{align}
so that
\begin{equation}
\label{022003-19}
\|\widetilde W^{(a)}(t)\|^2_{L^2(\bbR)}+\hat
  c \int_0^t\|\widetilde W^{(a)}(s)\|_{{\cal H}_a}^2ds=\|\bar
  W_0\|^2_{L^2(\bbR)},\quad t\ge0,\,a>0.
\end{equation}
Letting $a\to0+$ we conclude, from
  \eqref{022003-19} and  \eqref{011402-19} that
%upon an  application of the Fatou lemma, that
\begin{equation}
\label{022003-19a}
\|\widetilde W(t)\|^2_{L^2(\bbR)}+\hat
  c \int_0^t\|\widetilde W(s)\|_{{\cal H}_0}^2ds\le \|\bar
  W_0\|^2_{L^2(\bbR)},\quad t\ge0,
\end{equation}
which implies part (i) of Definition \ref{df011803-19}.

%\subsection{The case of an arbitrary temperature $T\ge0$}
When $T\neq 0$, let us set
\begin{equation}
\label{Wtya}
W^{(a)}(t,y):=\bbE\left[\bar W_0\left(\eta^o_a( t;y)\right),\, t<\hat  {\frak u}^a_{y,{\frak
  f}}\right]+T\bbP[ t\ge \hat  {\frak u}^a_{y,{\frak
  f}}],\quad (t,y)\in\bbR_+\times \bbR_*,
\end{equation}
where 
$ \bar W_0$ is given by \eqref{barW}.
It follows from \eqref{013101-19a}   that $W^{(a)}$ satisfies
\begin{align}
\label{013101-19aa}
&
\partial_t W^{(a)}(t,y)=\hat c L_aW^{(a)}(t,y)+\hat c p_-\int_{[yy'<0]}q^{(a)}_\beta(y-y')[ W^{(a)}(t,-y')- W^{(a)}(t,y')]dy'\nonumber\\
&
+\hat c\frak g_0\int_{[yy'<0]}q^{(a)}_\beta(y-y')[T- W^{(a)}(t,y')]dy',
\end{align}
while \eqref{012603-19} and \eqref{011402-19} imply
\begin{equation}
\label{011402-19a}
\lim_{a\to0+} W^{(a)}(t,y)
= W(t,y),\quad (t,y)\in \bar\bbR_+\times\bbR_*
\end{equation}
and
\begin{equation}
\label{Wty}
W(t,y)=\bbE\left[ \bar W_0\left(\eta^o( t;y)\right),\, t<\hat  {\frak u}_{y,{\frak
  f}}\right]+T\bbP[ t\ge \hat  {\frak u}_{y,{\frak
  f}}],\quad (t,y)\in\bbR_+\times \bbR_*.
\end{equation}
%where $\bar W_0$ is given by \eqref{barW}.
Part (i) of Definition \ref{df011803-19} is a direct
conclusion from the following.
\begin{prop}
\label{prop011903-19}
If $\bar W_0\in C_b(\bbR_*)\cap L^1(\bbR)$, then  
$W\in L^\infty_{\rm loc}([0,+\infty),L^2(\bbR))$ and~$W-T\in L^2_{\rm
  loc}([0,+\infty);{\cal H}_0)$.
\end{prop}
{\bf Proof.}
Let $\widetilde W^{(a)}(t):= W^{(a)}(t)-T$.
Multiplying
both sides of \eqref{013101-19aa} by $\widetilde W^{(a)}(t,y)$ and
integrating in the $y$ variable we obtain
\begin{align}
\label{022003-19zz}
&\|
  \bar W_0\|^2_{L^2(\bbR)}=\|
  W^{(a)}(t)\|^2_{L^2(\bbR)}+2T\left(\int_{\bbR}W^{(a)}(t,y)dy-\int_{\bbR}\bar
  W_0(y)dy\right)+\hat
  c \int_0^t\|\widetilde W^{(a)}(s)\|_{{\cal H}_a}^2ds,
\end{align}
hence
%Letting $a\to0+$ we conclude that
\begin{align}
\label{022003-19zz1}
&\| W^{(a)}(t)\|^2_{L^2(\bbR)}+\hat
  c \int_0^t\|\widetilde W^{(a)}(s)\|_{{\cal H}_a}^2ds\le \|
  \bar W_0\|^2_{L^2(\bbR)}+2T\left(\|W^{(a)}(t)\|_{L^1(\bbR)}+\|\bar W_0\|_{L^1(\bbR)}\right).
\end{align}
To estimate the $L^1$-norm in the right side, note that \eqref{Wtya} implies
\begin{align}
\label{Wty1}
|W^{(a)}(t,y)|\le \bbE\left[ |\bar W_0\left(\eta_a( t;y)\right)|+ |\bar W_0\left(-\eta_a( t;y)\right)|\right]+T\bbP[
  t\ge \hat  {\frak u}^a_{y,1}],
\end{align}
where $\eta_a(\cdot,y)$ is the Levy process with the generator
\eqref{La} starting at $y$, thus
%From here we conclude that
\begin{equation}
\label{Wty2}
\|W^{(a)}(t)\|_{L^1(\bbR)}\le 2\|\bar W_0\|_{L^1(\bbR)}+2T\int_{0}^{+\infty}\bbP[
  t\ge \hat  {\frak u}^a_{y,1}]dy\le 2\|\bar W_0\|_{L^1(\bbR)}+2T \bbE[ \sup_{s\in[0,t]}\eta_a(s ;0)].
\end{equation}
%Note that
%% $$
%% \bbP[
%%   t\ge \hat  {\frak u}^a_{y,1}]=\bbP[\sup_{s\in[0,t]}\eta_a(s;0)\ge
%%   y],
%% $$
%% therefore 
%$$
%\int_{0}^{+\infty}\bbP[
%  t\ge \hat  {\frak u}^a_{y,1}]dy=\bbE[ \sup_{s\in[0,t]}\eta_a(s ;0)].
%$$
%Therefore,
%\begin{align}
%\label{Wty3}
%\|W^{(a)}(t)\|_{L^1(\bbR)}\le 2\|\bar W_0\|_{L^1(\bbR)}+2T \bbE[ \sup_{s\in[0,t]}\eta_a(s ;0)].
%\end{align}
Since $\left(\eta_a(t,0)\right)_{t\ge0}$ is a martingale, we may use the  Doob maximal
inequality ($\left(\eta_a(t,0)\right)_{t\ge0}$ to see that there exists $C>0$ such that
\begin{align}
\label{sup-stable}
 \Big\{\bbE\Big[(\sup_{s\in[0,t]}\eta_a(s;0))^{\kappa}\Big]\Big\}^{1/\kappa}
\le C \left\{\bbE\left[
    \left|\eta_a(t
      ;0)\right|^{\kappa}\right]\right\}^{1/\kappa},
\end{align}
with $1<\kappa<1+1/\beta$. 
The argument in the proof of Lemma 5.25.7, p. 161 of \cite{sato} implies
$$
\limsup_{a\to0+}\bbE\left[
    \left|\eta_a(t
      ;0)\right|^{\kappa}\right]<+\infty,
$$
so that, in particular, $\bbE[
  \sup_{s\in[0,t]}\eta(s ;0)]<+\infty$.
Letting $a\to0^+$ in \eqref{Wty2}, we obtain  
\begin{align}
\label{Wty2a}
&
\lim_{a\to0+}\|W^{(a)}(t)\|_{L^1(\bbR)}=\|W(t)\|_{L^1(\bbR)}\le 2\|\bar W_0\|_{L^1(\bbR)}+2T \bbE[
  \sup_{s\in[0,t]}\eta(s ;0)] \\
&\le 2\|\bar W_0\|_{L^1(\bbR)}+2Tt^{\beta/(1+\beta)} \bbE[
  \sup_{s\in[0,1]}\eta(s ;0)].\nonumber
\end{align}
The last inequality follows from the  self-similarity of the
stable process $\left(\eta(t;0)\right)_{t\ge0}$. Now, we use \eqref{Wty2a} to bound the right side of \eqref{022003-19zz1},
and pass to the limit $a\to 0^+$ of that inequality to  finish the
proof.
\qed

\label{sec8a}

\appendix

\section{Proof of the existence part of Proposition \ref{prop013001-19}}

\label{app} 

We may assume without loss of generality that $T=0$ in \eqref{feb1408} and \eqref{feb1410},
since if 
$W(t,y,k)$ is a solution of \eqref{eq:8} in this case, with the
respective interface conditions, then
$W(t,y,k)+T$ solves the corresponding problem with a given temperature
$T>0$. Consider a semigroup of bounded operators on
$L^\infty(\bbR\times \bbT_*)$ defined by
\begin{align}
\label{010304}
&S_tW_0\left(y,k\right)
=  
e^{-\ga_0R(k)t}W_0\left(y-\bar{\om}'(k)t,k\right) 1_{[0,\bar{\om}'(k)t]^c}(y) +p_+(k)e^{-\ga_0R(k)t}
\\
&
\times W_0\left(y-\bar{\om}'(k)t,k\right)1_{[0,\bar{\om}'(k)t]}(y)
+p_-(k) e^{-\ga_0R(k)t}W_0\left(-y+\bar\om'(k) t,-k\right)1_{[0,\bar \om'(k)t]}(y),\nonumber
\end{align}
with $W_0\in L^\infty(\bbR\times
\bbT_*)$, $t\ge0$  and $(y,k)\in \bbR_*\times \bbT_*$. 
%The family $(S_t)_{t\ge0}$ is a semigroup of bounded operators on
%$B_b(\bbR\times \bbT_*)$.  
Note that if $W_0$ is continuous on $\bbR_*\times \bbT_*$, then
$S_tW_0(y,k)$ satisfies the interface conditions~\eqref{feb1408} and \eqref{feb1410} (with $T=0$) for all $t>0$, %and  
%$S_tW_0\in{\cal C}_0$ for any~$t>0$, 
so that $S_t$ maps ${\cal
  C}_0$ to ${\cal C}_0$, for any $t\ge0$ fixed. 
  In addition, $(S_t)_{t\ge0}$
is a $C_0$-semigroup on ${\cal C}_0$, with the supremum norm, satisfying
\begin{equation}
\label{010204-19}
D_t[S_tW_0(y,k)]=-\ga_0 R(k) S_tW_0(y,k),
\end{equation}
together with the interface condition \eqref{feb1408aa} and the initial condition
$$
\lim_{t\to0} S_tW_0(y,k)=W_0(y,k),\quad (y,k)\in\bbR_*\times\bbT_*.
$$
Using this semigroup, we can rewrite equation \eqref{eq:8} in the mild
formulation
\begin{equation}
  \label{eq:8c}
 W(t,y,k) = S_tW_0(y,k)+\ga_0\int_0^tS_{t-s} {\cal R} W(s,y,k)ds, 
\quad (t,y,k)\in\bbR_+\times \bbR_*\times\bbT_*,
\end{equation}
with
\begin{equation}
\label{calR}
{\cal R}F(y,k):=\int_{\bbT}R(k,k')F(y,k')dk',\quad F\in L^\infty(\bbR\times\bbT).
\end{equation}
The solution of \eqref{eq:8c} with $W_0\in {\cal C}_0$ %$L^\infty(\bbR\times \bbT_*)$
can be written as the Duhamel series
\begin{equation}
  \label{eq:8cb}
 W(t,y,k) =\sum_{n=0}^{+\infty}S^{(n)}(t,y,k),
\quad (t,y,k)\in\bbR_+\times \bbR_*\times\bbT_*,
\end{equation}
where 
\begin{align*}
 &
S^{(0)}(t,y,k) :=S_tW_0(y,k),\\
&
S^{(n)}(t,y,k) :=\ga_0^n\int_{\Delta_n(t)}S_{t-s_{1}} {\cal
  R}S_{s_{1}-s_{2}}\ldots {\cal R}S_{s_{n-1}-s_n} {\cal
  R}S_{s_n} W_0(y,k)ds_{1,n},\quad n\ge1,
\end{align*}
and 
$$
\Delta_n(t):=[t\ge s_1\ge\ldots\ge s_n\ge0],\quad
ds_{1,n}:=ds_1\ldots ds_n.
$$
Since $W_0$ is bounded, the series is uniformly
convergent on any $[0,t]\times\bbR\times \bbT_*$.  Moreover, 
if~$W_0\in {\cal C}_0$, then  $S_sW_0\in {\cal C}_0$ for all $s\ge 0$ and 
the function ${\cal R} S_sW_0$ is bounded and continuous in $\bbR_*\times
\bbT_*$, though it need not satisfy \eqref{feb1408aa}. On the other hand,
the function
$S_{t-s}{\cal
  R} S_sW_0$ satisfies the interface condition \eqref{feb1408aa} for
all $s\in[0,t]$, thus so does
$$
S^{(1)}(t,y,k)=\int_0^t S_{t-s}{\cal
  R} S_sW_0(y,k)ds,
$$
%In addition, as a result of the  integration ({\bf why do you need this integration?})
%in the time domain, we know that
and $S^{(1)}(t,\cdot,\cdot)\in {\cal C}_0$ for each $t\ge0$. A similar
argument shows that  $S^{(n)}(t,\cdot,\cdot)\in {\cal C}_0$ for 
all~$t\ge0$ and $n\ge1$. Hence,  $W(t,\cdot)$ defined by  the series \eqref{eq:8c}
belongs to  ${\cal C}_0$ for each $t>0$. One can also verify easily
that both
\eqref{eq:8a} and \eqref{010102-19} hold. 
%since $W_0\in C^{1}_b(\bbR_*\times\bbT_*)$. 
Thus, $W(t,y,k)$ is a
solution of \eqref{eq:8} in the sense of Definition \ref{df013001-19},
which ends the proof of the existence part of Proposition \ref{prop013001-19}.\qed

\section{Proof of Theorem \ref{thm013112}}

\label{appA}

\subsection{Preliminaries on the Skorokhod space $D[0,+\infty)$}\label{sec:appB1}

Let us denote by $D[0,+\infty)$ the space of the cadlag functions,~
see \cite{billingsley}. 
The 
$J_1$-topology on $D[0,+\infty)$ is induced by the metric 
$$
\rho_\infty(X_1,X_2):=\int_0^{+\infty}(\rho_T(X_1,X_2)\wedge 1)e^{-T}
dT,\quad X_1,X_2\in D[0,+\infty),
$$ 
where
\begin{equation}
\label{rho}
\rho_T(X_1,X_2):=\inf_{\la\in\Lambda_T}\max\Big\{\sup_{t\in[0,T]}|\la(t)-t|,
  \sup_{t\in[0,T]}|X_1\circ\la(t)-X_2(t)|\Big\},\quad X_1,X_2\in D[0,T].
\end{equation}
Here $D[0,T] $ is the space of cadlag functions on $[0,T]$
and 
$\Lambda_T$ is the collection of homeomorphisms $\la:[0,T]\to[0,T]$ 
such that $\la(0)=0$, $\la(T)=T$.
Theorem 16.1 of \cite{billingsley} says that for a sequence
$(X_n)_{n\ge1}$ of cadlag functions: $X_n\to_{\rho_\infty}X$ iff there
exists a sequence of strictly increasing homeomorphisms $\la_n:[0,+\infty)\to [0,+\infty)$
such that for each $N>0$ we have
\begin{equation}
\label{hom}
 \lim_{n\to+\infty}\sup_{0\le t\le N}|t- \la_n(t)|=0\quad \mbox{and}\quad \lim_{n\to+\infty}\sup_{0\le t\le N}|X_n(t)-X\circ \la_n(t)|=0.
\end{equation}
The $M_1$-topology on $D[0,+\infty)$ is defined as follows.
For a given  $X\in D[0,T]$,
let $\Gamma_X$ be the graph of $X$:
\begin{equation}
\label{GX}
\Gamma_X:=[(t,z):t\in[0,T],\, z=c X(t-)+(1-c)X(t)],\mbox{ for some
}c\in[0,1]].
\end{equation} 
We define an order on $\Gamma_X$ by letting $(t_1,z_1)\le(t_2,z_2)$
iff
$t_1< t_2$, or $t_1=t_2$ and 
$$
|X(t_1-)-z_1|\le |X(t_1-)-z_2|.
$$ 
Denote
by $\Pi(X)$ the set of all continuous mappings $\gamma=(\gamma^{(1)},\gamma^{(2)}):[0,1]\to\Gamma_X$ that
are non-decreasing, i.e. $t_1\le t_2$ implies that $\gamma(t_1)\le
\gamma(t_2)$.
The metric ${\rm d}_T(\cdot,\cdot)$ is defined as follows:
$$
 \mbox{d}_T(X_1,X_2):=\inf[\|\gamma^{(1)}_1-\gamma^{(1)}_2\|_\infty\vee \|\gamma^{(2)}_1-\gamma^{(2)}_2\|_\infty,\,\gamma_i=(\gamma_i^{(1)},\gamma_i^{(2)})\in \Pi(X_i),\,i=1,2].
$$
This metric induces the $M_1$-topology in $D[0,T]$, see
\cite{whitt}, Theorem 13.2.1.  The  corresponding topology in
$D[0,+\infty)$ is defined by
\begin{equation}
\label{dinfty}
{\rm d}_\infty(X_1,X_2):=\int_0^{+\infty} 
({\rm d}_T(X_1,X_2)\wedge 1)e^{-T}
dT,\quad X_1,X_2\in D[0,+\infty).
\end{equation}
Then (see Theorem 6.3.2 of \cite{whitt-s}), we have
$$
{\rm d}_\infty(X_1,X_2)\le \rho_\infty(X_1,X_2) ,\quad X_1,X_2\in D[0,+\infty).
$$
%\red{This seems a misprint}
%One can show, see Theorem 6.3.2 of \cite{whitt-s}, that
%$$
%{\rm d}(X_1,X_2)\le \rho(X_1,X_2) ,\quad X_1,X_2\in D[0,T].
%$$
Let $\bar{T}_y, \underbar{T}_y:D[0,+\infty)\to[0,+\infty]$ be   
$$
\bar{T}_y(X):=\inf[t>0:\,X(t)>y],\quad
\underbar{T}_y(X):=\inf[t>0:\,X(t)<y],\quad X\in D[0,+\infty),
$$
%with the convention that the infimum of an empty set is $+\infty$, 
and for
$a\ge0$, let
~$\theta_a:D[0,+\infty)\to D[0,+\infty)$ be
\begin{equation}
\label{shift}
\theta_a(X)(t):=X(t+a),\quad t\ge0.
\end{equation}
Finally, we use the notation
${\cal M}, {\cal M}':D[0,+\infty)\to D[0,+\infty)$ for
\begin{equation}
\label{cM}
{\cal M}(X)(t):=\sup_{0\le s\le t}X(s),~~
%\end{equation} 
%and 
%\begin{equation}
%\label{cMp}
{\cal M}'(X):=-{\cal M}(-X)=\inf_{0\le s\le t}X(s). 
\end{equation}  
Both of these mappings are  
$J_1$-continuous, see Theorem 7.4.1 %p. 172 
of~\cite{whitt-s}.

\subsection*{Joint convergence of $((Z_{N}(t), T_N(t))_{t\ge0},{\frak t}_{y,1}^N)$}

To simplify the notation, we suppress writing $k$ and 
$y$ in the notation of the processes, denoting them by
$\left(Z_N(t)\right)_{t\ge0}$ and $\left(\zeta(t)\right)_{t\ge0}$, respectively. 
For $y>0$, we introduce the consecutive times the trajectory  
$Z_N(t)$ crosses the level $y$: 
$
{\frak t}_{y,1}^{N}:=\bar{T}_y(Z_{N}),
$
 and
$
{\frak t}^{N}_{y,2}:=\underbar{T}_y\circ \theta_{\bar{T}_y(Z_{N})}(Z_{N}).
$  
Having defined ${\frak t}_{y,2m-1}^{N}$, ${\frak t}_{y,2m}^{N}$ for some $m\ge1$,
we set
\begin{align}
\label{010412ba}
{\frak t}^{N}_{y,2m+1}:=\bar{T}_y\circ \theta_{{\frak t}^{N}_{y,2m}}(Z_{N}),\quad
 {\frak t}^{N}_{y,2m+2}:=\underbar{T}_y\circ \theta_{{\frak t}^{N,o}_{y,2m+1}}(Z_{N}).
\end{align}
We introduce the crossing times for $y<0$ similarly, as well as
the crossing times $\left({\frak t}_{y,m}\right)_{m\ge1}$ for
the process $\left(\zeta(t)\right)_{t\ge0}$. The conclusion of the
theorem is equivalent to proving that 
$$
(\left(Z_{N}(t), \frak T_N(t)\right)_{t\ge0},({\frak t}_{y,m}^N)_{m\ge1},(Z_N({\frak
  t}_{y,m}^{N}))_{m\ge1})
$$ converge in law, as $N\to+\infty$, to  
$$
((\zeta(t),\tau(t))_{t\ge0},({\frak
  t}_{y,m})_{m\ge1},(\zeta({\frak t}_{y,m}))_{m\ge1}).
$$ 
We prove this by an induction argument on $m$. Let us fix
$y>0$. 
Since~$(Z_{N}(t))_{t\ge0}$ converges in  law to $(Z(t))_{t\ge0}$ the
 processes 
 \begin{equation}
 \label{023112}
 M_{N}(t):={\cal M}(Z_{N})(t),\quad t\ge0.
 \end{equation}
 are likewise convergent to $M(t):={\cal M}(\zeta)(t)$, ${t\ge0}$.
Since 
$\bbP[\zeta(t)=\zeta(t-)]=1$ for each $t\ge0$ (see, for example, Proposition
1.2.7 of \cite{bertoin}) we have
$\bbP[M(t)=M({t-})]=1$, and, as a result, the finite-dimensional marginals of
$(M_N(t))_{t\ge0}$ converge to those of $(M(t))_{t\ge0}$, see, for instance,
Theorem 3.16.6 of \cite{billingsley}. Since, in addition the law
of $M(t)$ is absolutely continuous, see e.g. Theorem 4.6 of
\cite{KMR}, we have 
\begin{equation}
 \label{023112a}
\bbP[{\frak t}_{y,1}^N\le t]=\bbP[M_{N}(t)>y]\to \bbP[M(t)>y]=\bbP[{\frak t}_{y,1}\le t],\quad
\mbox{ as }N\to+\infty,
\end{equation}
for any $y,t>0$. 
% As a result we have
% $$
% \bbP[s<{\frak t}_{y,1}^N\le t]\to \bbP[s<{\frak t}_{y,1}\le t],\quad
% \mbox{ as }N\to+\infty
% $$
% for any $y,t>0$.
Hence, both   marginals of  $((Z_{N}(t),\frak T_N(t))_{t\ge0},{\frak t}_{y,1}^N)$
converge in law towards the respective laws of the marginals of
$((\zeta(t),\tau(t))_{t\ge0},{\frak t}_{y,1})$. We need to show the joint convergence.

Let us recall that ${\cal D}_2:=D([0,+\infty);\bbR\times\bar\bbR_+)$, and let $F: {\cal D}_2\times \bar\bbR_+\to\bbR$ be a bounded
and continuous function.
We need to show that, see Theorem 1.1.1(ii) of \cite{stroock-varadhan}, 
\begin{equation}
\label{012411-18}
\lim_{N\to+\infty}\bbE F ((Z_{N}(t),\frak T_N(t))_{t\ge0},{\frak t}_{y,1}^N)=\bbE F((\zeta(t),\tau(t))_{t\ge0},{\frak t}_{y,1}).
\end{equation}
It is straightforward to check that it is suffices to prove (\ref{012411-18}) only for functions of the form~$F(\omega,t)=G(\omega)\psi(t)$, 
with a bounded continuous function
$G:{\cal D}_2\to\bbR$ and a compactly supported continuous function $\psi:\bar\bbR_+\to\bbR$:
\begin{equation}
\label{022411-18}
\lim_{N\to+\infty}\bbE \left[G ((Z_{N}(t),\frak T_N(t))_{t\ge0})\psi({\frak
    t}_{y,1}^N)\right]=\bbE \left[G ((\zeta(t),\tau(t))_{t\ge0})\psi({\frak
    t}_{y,1}) \right].
\end{equation}
To this end, suppose that $t>0$ and $y>0$ are fixed, and
consider the function 
$$
{\frak F}_t(X,S)= G(X,S)1_{(y,+\infty)}(\pi_t\circ {\cal M}(X)),
$$ 
where
%$F:{\cal D}_2\to\bbR$ is bounded and continuous, $[0,+\infty)\ni t\mapsto
%(X(t),S(t))\in\bbR\times\bar\bbR_+$ is a cadlag function and
$\pi_t(X):=X(t)$. 
We claim that the set ${\rm
  Disc}({\frak F}_t)$ of
discontinuities of the function ${\frak F}_t$ has zero measure under the law of
$(\zeta(t),\tau(t))_{t\ge0}$. 
First, observe that 
${\rm Disc}(\pi_t\circ {\cal M})$ is of zero measure. 
Indeed, if $X\in {\rm Disc}(\pi_t\circ {\cal M})$, then  ${\cal
  M}(X)\in {\rm Disc}(\pi_t)$. Theorem~16.6(i)
of \cite{billingsley} implies that then ${\cal
  M}(X)(t-)\not={\cal
  M}(X)(t)$, which implies
$X(t-)\not=X(t)$,  and the latter set has zero measure under the law of
$(\zeta(t))_{t\ge0}$. On the other hand, 
if~$X\in {\rm Disc}(1_{(y,+\infty)}\circ\pi_t\circ {\cal M})$ 
but~$X\not\in {\rm Disc}(\pi_t\circ {\cal M})$, it follows that 
$\pi_t\circ {\cal M}(X)=y$, which is a set of measure zero, 
%since
%its measure is the same as
%$
%\bbP[M_t=y]=0,
%$ see  see e.g. 
by Theorem 4.6 of~\cite{KMR}.

The above implies that the set of discontinuities of ${\frak F}_t$ has measure zero. 
Hence, by Theorem 2.7 of \cite{billingsley}, we have
\begin{equation}
\label{022411-18a-bis}
\lim_{N\to+\infty}\bbE \left[{\frak F}_t
(  (Z_{N}(t),\frak T_N(t))_{t\ge0})\right]=\bbE \left[{\frak F}_t( (\zeta(t),\tau(t))_{t\ge0})\right],
\end{equation}
or equivalently
\begin{equation}
\label{022411-18b}
\lim_{N\to+\infty}\bbE \left[G (Z_{N}(t),\frak T_N(t))_{t\ge0} )1_{[0,t]}({\frak
    t}_{y,1}^N)\right]=\bbE \left[G ( (\zeta(t),\tau(t))_{t\ge0}) 1_{[0,t]} ({\frak
    t}_{y,1}) \right]
\end{equation}
for any $t>0$.
The above implies  
\begin{equation}
\label{022411-18c}
\lim_{N\to+\infty}\bbE \left[G (Z_{N}(t),\frak T_N(t))_{t\ge0} )1_{(s,t]}({\frak
    t}_{y,1}^N)\right]=\bbE \left[G ( (\zeta(t),\tau(t))_{t\ge0}) 1_{(s,t]} ({\frak
    t}_{y,1}) \right]
\end{equation}
for any $0\le s<t$. We can approximate (in the supremum norm) any compactly supported,
continuous function $\psi$ by step functions of the form
$\sum_{i=1}^Ic_i1_{(s_i,t_i]}$, $s_i<t_i$. This ends the proof of \eqref{022411-18}.

\subsection*{Convergence of $(\left(Z_{N}(t), \frak T_N(t)\right)_{t\ge0},{\frak t}_{y,1}^N, Z_{N}({\frak t}_{y,1}^N))$}

By the already proved part of the theorem  we  know that $(\left(Z_{N}(t), \frak T_N(t)\right)_{t\ge0},{\frak t}_{y,1}^N)$ 
converges in law to $(\left(\zeta(t), \tau(t)\right)_{t\ge0},{\frak t}_{y,1})$.
According to the Skorokhod embedding theorem, see e.g.  Theorem I.6.7 of \cite{billingsley}, we can assume that there exists 
a realization of the sequence of the processes 
$(\left(Z_{N}(t), \frak T_N(t)\right)_{t\ge0},{\frak t}_{y,1}^N)$,
over a certain probability space $(\Om,{\cal F},\bbP)$, such that
\begin{equation}
\label{012812}
\lim_{N\to+\infty}\rho_\infty((Z_N,\frak T_N),(\zeta,\tau))=0\quad \mbox{and}\quad \lim_{N\to+\infty}|{\frak t}_{y,1}^{N}-{\frak t}_{y,1}|=0,\quad \mbox{a.s.}
\end{equation}
and let 
$$
{\frak Z}^N_m:=Z_N({\frak t}_{y,m}^{N}),
\quad {\frak Z}_m:=\zeta({\frak t}_{y,m}),\quad N,m\ge1.
$$
\begin{lm}
\label{lm012912}
For the above realization of the sequence  $(\left(Z_{N}(t), \frak T_N(t)\right)_{t\ge0},{\frak t}_{y,1}^N)$,   we have
\begin{equation}
\label{012012}
\lim_{N\to+\infty}|{\frak Z}^N_1-{\frak Z}_1|=0,\quad \mbox{a.s.}
\end{equation}
\end{lm}
{\bf Proof.} Assume that $y>0$.
Thanks to \eqref{012812}, 
there exist a sequence $\la_N$ of increasing homeomorphisms of $[0,+\infty)$ such that
for any $T>0$, we have, a.s:
\begin{align}
\label{022812}
&
\lim_{N\to+\infty}\sup_{t\in[0,T]}|Z_N(t)-\zeta\circ\la_N(t)|=0,~~
\lim_{N\to+\infty}\sup_{t\in[0,T]}|t-\la_N(t)|=0,~~ %\nonumber\\
%&
\lim_{N\to+\infty}\bar{T}_y(Z_{N})=\bar{T}_y(\zeta),%\nonumber
\end{align}
%all of them a.s.
%From this we get
hence
\begin{equation}
\label{032812}
\lim_{N\to+\infty}\la_N\left(\bar{T}_y(Z_{N})\right)=\bar{T}_y(\zeta),\quad \mbox{a.s.}
\end{equation}
We claim that for $\bbP$ a.s. $\om\in\Om$ there exists $N_0(\om)$ such that
\begin{equation}
\label{010401-19}
\la_N\left(\bar{T}_y(Z_{N}(\om))\right)=\bar{T}_y(\zeta(\om)),\quad
N\ge N_0.
\end{equation}
Indeed, consider  two cases.

{\em Case (1)} For a given $\om\in \Om$ there exists an infinite sequence $N_k$ such that
\begin{equation}
\label{042812}
\la_{N_k}\left(\bar{T}_y(Z_{N_k}(\om))\right)>\bar{T}_y(\zeta(\om)).
\end{equation}
Then, there exists a (random) sequence $(t_{N_k})$ that satisfies
$$
\bar{T}_y(Z_{N_k}(\om)) >t_{N_k}>\la_{N_k}^{-1}\left(\bar{T}_y(\zeta(\om))\right).
$$
From \eqref{032812}, we conclude that
$$
\lim_{k\to+\infty} t_{N_k}=\bar{T}_y(\zeta(\om)),
$$
therefore, by  \eqref{022812}, we have
$$
\lim_{k\to+\infty} \la_{N_k}(t_{N_k})=\bar{T}_y(\zeta(\om)).
$$
From the right continuity of $\zeta$ and (\ref{042812}), we deduce  that then
 \begin{equation}
\label{052812}
\lim_{k\to+\infty}\zeta\left(\la_{N_k}\left(t_{N_k}\right)\right)=\zeta\left(\bar{T}_y(\zeta(\om))\right).
\end{equation}
From the first equality in \eqref{022812} we infer that
\begin{equation}
\label{052812aaa}
\lim_{k\to+\infty}Z_{N_k}\left(t_{N_k}\right)=\zeta\left(\bar{T}_y(\zeta(\om))\right).
\end{equation}
However, since $\bar{T}_y(Z_{N_k}(\om)) >t_{N_k}$ we have
$$
Z_{N_k}\left(t_{N_k}\right)\le y,
$$
which would imply that
\begin{equation}\label{may602}
\zeta\left(\bar{T}_y(\zeta(\om))\right)\le y,
\end{equation}
hence
\begin{equation}\label{may604}
\omega\in{\cal N}_0=\{\omega:~\zeta\left(\bar{T}_y(\zeta(\om))\right)=y\}.
\end{equation}
According to Corollary 2.2 of \cite{millar}, the probability of  ${\cal N}_0$ is zero.  
%there exists a null event ${\cal
%  N}_0$ such that $\om\in{\cal N}_0$.

\bigskip

{\em Case (2)}. For a given $\om\in \Om$ there exist infinitely many
$N_k$-s such that
\begin{equation}
\label{062812}
\la_{N_k}\left(\bar{T}_y(Z_{N_k}(\om))\right)< \bar{T}_y(\zeta(\om)),
\end{equation}
so that
 \begin{equation}
\label{052812a}
\lim_{k\to+\infty}\zeta\left(\la_{N_k}\left(\bar{T}_y(Z_{N_k}(\om))\right)\right)=\zeta\left(\bar{T}_y(\zeta(\om))-\right)\le y.
\end{equation}
On the other hand we have
$$
Z_{N_k}\left(\bar{T}_y(Z_{N_k}(\om))\right)\ge y,
$$
and, by \eqref{022812},
\begin{equation}
\label{062812a}
\lim_{k\to+\infty}\left|\zeta\left(\la_{N_k}\left(\bar{T}_y(Z_{N_k}(\om))\right)\right)-
Z_{N_k}\left(\bar{T}_y(Z_{N_k}(\om))\right)\right|=0,
\end{equation}
so that
\begin{equation}
\label{062812b}
\lim_{k\to+\infty}\zeta\left(\la_{N_k}\left(\bar{T}_y(Z_{N_k}(\om))\right)\right)= 
\lim_{k\to+\infty} Z_{N_k}\left(\bar{T}_y(Z_{N_k}(\om))\right)\ge y.
\end{equation}
Comparing to \eqref{052812a}, we see that
\begin{equation}
\label{062812c}
\lim_{k\to+\infty}\zeta\left(\la_{N_k}\left(\bar{T}_y(Z_{N_k}(\om))\right)\right)= 
\lim_{k\to+\infty} Z_{N_k}\left(\bar{T}_y(Z_{N_k}(\om))\right)=y.
\end{equation}
Therefore,
from \eqref{052812a} and \eqref{062812c} we get
\begin{equation}
\label{052812b}
\lim_{k\to+\infty}\zeta\left(\la_{N_k}\left(\bar{T}_y(Z_{N_k}(\om))\right)\right)=\zeta\left(\bar{T}_y(\zeta(\om))-\right)=
y.
\end{equation}
Hence, we either have
\begin{equation}\label{may606}
\omega\in {\cal N}_1=[\omega:~\zeta\left(\bar{T}_y(\zeta(\om))-\right)=y,~\zeta\left(\bar{T}_y(\zeta(\om))\right)>y],
\end{equation}
an event that has probability zero by Proposition, on p. 695 of 
\cite{lamperti}, or $\omega\in{\cal N}_0$. We conclude that  \eqref{010401-19} holds. 
%
%If $\zeta\left(\bar{T}_y(\zeta(\om))\right)>y$, then
%by a Proposition, on p. 695 of 
%\cite{lamperti}  there exists an event
%${\cal N}_1$ such that $\bbP[{\cal N}_1]=0$ and $\om\in {\cal N}_1$.
%If, on the other hand $\zeta\left(\bar{T}_y(\zeta(\om))\right)=y$,
%then by the result of \cite{millar} invoked in the previous case we
%conclude that $\om\in {\cal N}_0$. In any case though $\om\in {\cal
%  N}_0\cup{\cal N}_1$ and \eqref{010401-19} is in force.
This however, obviously implies \eqref{012012}, as
$$
{\frak Z}_1^N=Z_N\left(\bar{T}_y(Z_{N}(\om))\right)\quad\mbox{and}\quad {\frak Z}_1=\zeta\left(\bar{T}_y(\zeta(\om))\right),
$$
finishing the proof. \qed

\bigskip

\subsection*{Generalization to subsequent exit times -- the end of the
  proof of Theorem \ref{thm013112}}

\begin{cor}
\label{cor012912}
Under the assumptions of Lemma $\ref{lm012912}$, for any $y\in\bbR$ we have
\begin{equation}
\label{012912}
\lim_{N\to+\infty}\rho_\infty(\theta_{{\frak t}_{y,1}^{N}}(Z_N),\theta_{{\frak t}_{y,1}}(\zeta))=0,\quad \mbox{a.s.}
\end{equation}
\end{cor}
{\bf Proof.}
%We  equalities of \eqref{022812}, as well as  \eqref{012812}, are in force.  
Define the following increasing homeomorphism of $[0,+\infty)$:
$$
\tilde \la_N(t):=\la_N\left(\bar{T}_y(Z_{N}(\om))+t\right)-\la_N(\bar{T}_y(Z_{N}(\om))),\quad t\ge0.
$$
Thanks to the first two equalities in \eqref{022812}, for any $T>0$ we have
\begin{align}
\label{022812b}
&
\lim_{N\to+\infty}\sup_{t\in[0,T]}|Z_N(\bar{T}_y(Z_{N}(\om))+t)-\zeta(\tilde \la_N(t)+\la_N(\bar{T}_y(Z_{N}(\om))))|=0,\\
&
\lim_{N\to+\infty}\sup_{t\in[0,T]}|t-\tilde \la_N(t)|=0.\nonumber
\end{align}
It follows from the argument in the proof of Lemma \ref{lm012912}  that there exists a $\bbP$-null set ${\cal N}$ such that for each $\om\not\in {\cal N}$ there exists $N_0$, for which   \eqref{042812} holds for all $N\ge N_0$. 
% $$
% \tilde \la_N(t)+\la_N\left(\bar{T}_x(Z_{N}(\om)\right)-\bar{T}_x(\zeta(\om))=\tilde \la_N(t).
% $$
% for all sufficiently large $N$-s. 
% In addition,
% \begin{align}
% \label{022812c}
% &
% \limsup_{N\to+\infty}\sup_{t\in[0,T]}\left|Z_N(\bar{T}_x(Z_{\cdot,N}(\om))+t)-Z\left(\bar{T}_x(Z_{\cdot}(\om))+\tilde \la_N(t)\right)\right|\nonumber\\
% &
% \le\limsup_{N\to+\infty}\sup_{t\in[0,T]}\left|Z_N(\bar{T}_x(Z_{\cdot,N}(\om))+t)-Z\left(\bar{T}_x(Z_{\cdot}(\om))+\tilde \la_N(t)\right)\right|\nonumber
% \end{align}
% We also have
From this equality we conclude that
\begin{align*}
 &
 \limsup_{N\to+\infty}\sup_{t\in[0,T]}|\theta_{\bar{T}_y(Z_{N}(\om) )}(Z_N)(t)-\theta_{\bar{T}_y(\zeta(\om))}(\zeta)(\tilde \la_N(t))|\\
 &
 =
   \limsup_{N\to+\infty}\sup_{t\in[0,T]}|Z_N(\bar{T}_y(Z_{,N}(\om))+t)-\zeta(\tilde \la_N(t)+\la_N(\bar{T}_y(Z_{\cdot,N}(\om))))|\\
&
 = \limsup_{N\to+\infty}\sup_{t\in[0,T]}|Z_N(\bar{T}_y(Z_{N}(\om))+t)-\zeta(\la_N(t+\bar{T}_y(Z_{N}(\om))))|=0
\end{align*}
for any $T>0$.
We have shown therefore that
\eqref{012912} holds.
\qed

\bigskip

Let
\begin{equation}
\label{tildeZ}
 Z'_{N}(t;\om):=\theta_{{\frak t}_{y,1}^N}(Z_N)(t),\quad 
\zeta'(t):=\theta_{{\frak t}_{y,1} }(\zeta)(t),\quad t\ge0
\end{equation}
and
$$
\tilde{\frak t}^{N}_{1,y}(\om)=\underbar{T}_y(
Z'_{N}(\om)),\quad \tilde{\frak t}_{y,1}(\om)=\underbar{T}_x( \zeta'(\om)).
$$
Note that
\begin{equation}
\label{013112}
{\frak t}^{N}_{y,2}(\om)={\frak t}^{N}_{y,1}(\om)+ \tilde{\frak t}^{N}_{y,1}(\om),\quad {\frak t}_{y,2}(\om)= {\frak t}_{y,1}(\om)+ \tilde{\frak t}_{y,1}(\om).
\end{equation}
Repeating the argument used in the proof of \eqref{023112a} we conclude that
$$
\bbP[\tilde{\frak t}^{N}_{y,1}\le t]\to \bbP[\tilde {\frak t}_{y,1}\le t],\quad
\mbox{ as }N\to+\infty.
$$
% \textcolor{red}{{\em This argument to be commented.}
% From our construction so far we know that
% $$
% \lim_{N\to+\infty}{\frak Z}_1^N ={\frak Z}_1,\quad \mbox{a.s.}
% $$
% The process 
% $
% \left(\tilde Z_{N}(t;\om)- {\frak Z}_1^N \right)_{t\ge0}
% $
% converges in law to 
% $
% \left(\tilde
% \zeta(t)- {\frak Z}_1 \right)_{t\ge0}.
% $
% This in particulal implies that
% $
% {\cal M}'\left( \tilde Z_{N}- {\frak Z}_1^N \right)
% $
% converges in law to
% $
% {\cal M}'\left(\tilde
% \zeta- {\frak Z}_1 \right).
% $
% In consequence, for any $t\ge0$ we have
% $$
% {\cal M}'\left( \tilde Z_{N}- {\frak Z}_1^N \right)(t)\Rightarrow {\cal M}'\left(\tilde
% \zeta- {\frak Z}_1 \right)(t),
% $$
% as 
% $$
% \bbP\left[ {\cal M}'\left(\tilde
% \zeta- {\frak Z}_1 \right)(t-)\not=  {\cal M}'\left(\tilde
% \zeta- {\frak Z}_1 \right)(t)\right]=0.
% $$
% Hence,
% $$
% {\frak Z}_1^N +{\cal M}'\left( \tilde Z_{N}- {\frak Z}_1^N \right)(t)\Rightarrow {\frak Z}_1+ {\cal M}'\left(\tilde
% \zeta- {\frak Z}_1 \right)(t).
% $$
% We can write
% \begin{align*}
% &
% \bbP[\tilde{\frak t}^{N}_{y,1}\le t]=\bbP\left[{\frak Z}_1^N +{\cal
%   M}'\left( \tilde Z_{N}- {\frak Z}_1^N \right)(t)<y\right] \to
% \\
% &
% \bbP\left[ {\frak Z}_1+ {\cal M}'\left(\tilde
% \zeta- {\frak Z}_1 \right)(t)<y\right]=\bbP[\tilde {\frak t}_{y,1}\le
%   t].
% \end{align*}
% {\em the end of the commented argument}}
This also proves  the tightness of the   random
elements~$((Z_{N}(t),\frak T_N(t))_{t\ge0},
{\frak t}^{N}_{y,1}, {\frak t}^{N}_{y,2},{\frak Z}_1^N)$, $N\ge1$. Using the
same  argument as in the proof of \eqref{012411-18} we can reduce the
proof of the convergence in law to
showing that  
for any bounded and continuous functions
$F:{\cal D}_2\times\bar \bbR_+\times\bbR \to\bbR$ and compactly supported continuous $\psi:\bar\bbR_+\to\bbR$ we have
\begin{equation}
\label{022411-18a}
\lim_{N\to+\infty}\bbE \left[F ((Z_{N}(t),\frak T_N(t))_{t\ge0}, {\frak
    t}^{N}_{y,1},{\frak Z}_1^N)\psi( \tilde{\frak
    t}^{N}_{y,1})\right]=\bbE \left[F (\zeta(t),\tau(t))_{t\ge0},
  {\frak t}_{y,1},{\frak Z}_1)\psi( \tilde{\frak t}_{y,1}) \right].
\end{equation}
Suppose that $t>0$ and
consider the function 
$$
{\frak F}_t:(X,S,s,z)\mapsto F(X,S,s)1_{(y,+\infty)}(\pi_t\circ {\cal
  M}'(\theta_s(X)),
$$ where $(X,S,s,z)\in{\cal D}_2\times\bar \bbR_+\times\bbR $, 
$F$ is as above and
$\pi_t(X):=X(t)$, and let $Q$ be  the law 
of~$\left( (\zeta(t),\tau(t))_{t\ge0},  {\frak t}_{y,1},{\frak Z}_1\right)$. 
We claim that the set ${\rm
  Disc}({\frak F}_t)$ of
discontinuities of the function ${\frak F}_t$ is $Q$-null.
Indeed, first observe that  the set $D$ of discontinuities of 
$$
(X,S,s,z)\mapsto \pi_t\circ {\cal M}'(\theta_s(X))
$$ is  $Q$-null.
If $(X,S,s,z)\in D$, then  ${\cal
  M}'(\theta_s(X))\in {\rm Disc}(\pi_t)$. According to Theorem 16.6(i)
of~\cite{billingsley}, this is equivalent to ${\cal
  M}'(\theta_s(X) )(t-)\not={\cal
  M}'(\theta_s(X))(t)$. However, this set is %e set of those $X$-s is
contained in
$$
[(X,s):\,X(s+t-)\not=X(s+t)].
$$ 
The $Q$-probability of the latter is
\begin{equation}
\label{043112-18}
\bbP\left[
\zeta( {\frak t}_{y,1}+t-)\not=\zeta( {\frak
    t}_{y,1}+t)\right]=\bbE\left\{\bbP\left[\zeta( {\frak
      t}_{y,1}+t-)\not=\zeta( {\frak t}_{y,1}+t)\Big|{\cal F}_{{\frak
      t}_{y,1}}\right]\right\}.
\end{equation}
The strong Markov property implies that the process $\left(\zeta( {\frak
    t}_{y,1}+t)-{\frak Z}_1\right)_{t\ge0}$ is independent of the $\si$-algebra ${\cal F}_{{\frak
      t}_{xy1}}$ corresponding to the stopping time ${\frak
    t}_{y,1}$, and the right  side of~\eqref{043112-18}
  equals
 \begin{equation}
\label{043112-18a}
\bbP\left[\zeta(t-)\not=\zeta(t)\right]=0.
\end{equation}
Suppose now that 
$(X,s)\in D'$, the discontinuity set of 
$$(X,S,s,z)\mapsto
1_{(y,+\infty)}\circ\pi_t\circ {\cal M}'(X\circ \theta_s)
$$ and 
$(X,S,s,z)\not\in D$, so  that 
$\pi_t\circ {\cal M}'(\theta_s(X))=y$. Its probability equals
$$
\bbP\Big[\inf_{{\frak t}_{y,1}\le u\le {\frak t}_{x,1}+t}\zeta(u)=y
\Big]=
\bbE\bbP[M'(t)=y-z]_{z={\frak Z}_1},
$$
where $M'(t)={\cal M}'(\zeta)(t)$. By symmetry, the expression in the
right equals
$$
\bbE\bbP[M(t)=z-y]_{z={\frak Z}_1}=0,
$$
 as the law of $M(t)$ is absolutely continuous.
It follows that the set of discontinuities of ${\frak F}_t$ is
null. Hence, see Theorem 2.7 of \cite{billingsley}, we have
\begin{equation}
\label{022411-18aa}
\lim_{N\to+\infty}\bbE \left[{\frak F}_t
  ((Z_{N}(t),\frak T_N(t))_{t\ge0}, {\frak t}^{N}_{y,1},{\frak Z}_1^N)\right]=\bbE \left[{\frak F}_t (\zeta(t),\tau(t))_{t\ge0}, {\frak t}_{y,1}, {\frak Z}_1)\right],
\end{equation}
or equivalently
\begin{equation}
\label{022411-18bb}
\lim_{N\to+\infty}\bbE \left[F ((Z_{N}(t),\frak T_N(t))_{t\ge0}, {\frak t}^{N}_{y,1}, {\frak Z}_1^N)1_{[0,t]}({\frak
    t}_{x,2}^{N})\right]=\bbE \left[F ((\zeta(t),\tau(t))_{t\ge0},
  {\frak t}_{y,1},{\frak Z}_1) 1_{[0,t]} ({\frak
    t}_{y,2}) \right]
\end{equation}
for any $t>0$.
The above implies that 
\begin{equation}
\label{022411-18cc}
\lim_{N\to+\infty}\bbE \left[F ((Z_{N}(t),\frak T_N(t))_{t\ge0} ,{\frak
    t}^{N}_{y,1},{\frak Z}_1^N)1_{(s,t]}({\frak
    t}_{y,2}^{N})\right]=\bbE \left[F ((\zeta(t),\tau(t))_{t\ge0},
  {\frak t}_{y,1},{\frak Z}_1) 1_{(s,t]} ({\frak
    t}_{y,2}) \right]
\end{equation}
for any $0\le s<t$. We can approximate (in the supremum norm) any compactly supported
function $\psi$ by step functions of the form
$\sum_{i=1}^Ic_i1_{(s_i,t_i]}$, $s_i<t_i$. This ends the proof of
\eqref{022411-18a}.

By the previous argument  we already know that the random elements
\begin{equation}
\label{010201-19}
\left((Z_{N}(t),\frak T_N(t))_{t\ge0},{\frak t}_{y,1}^{N},{\frak
    t}_{y,2}^{N},{\frak Z}_1^N\right)
\end{equation}
 converge in law to 
\begin{equation}
\label{020201-19}
\left((\zeta(t),\tau(t))_{t\ge0},{\frak t}_{y,1},{\frak t}_{y,2},
  {\frak Z}_1\right).
\end{equation}
According to the Skorokhod embedding theorem, we can assume that there
exist realizations of the random elements \eqref{010201-19}, \eqref{020201-19} such that
\begin{align}
\label{030201-19}
&
\lim_{N\to+\infty}\rho_\infty((Z_N,\frak T_N),(\zeta,\tau))=0\quad \mbox{and}\\
&
\sum_{i=1}^2\lim_{N\to+\infty}|{\frak t}_{y,i}^{N}-{\frak
  t}_{y,i}|=0,\quad \lim_{N\to+\infty}|{\frak Z}_1^N-
  {\frak Z}_1|=0,\quad 
\mbox{a.s.}\nonumber
\end{align}
By Corollary \ref{cor012912} we have
\begin{equation}
\label{tildeZ1}
\lim_{N\to+\infty}\rho_\infty(Z'_{N}, \zeta')=0,
\end{equation}
where $\tilde Z_{N}$ and $\tilde \zeta$ are defined by \eqref{tildeZ}.
We can repeat the argument used in the proof of Lemma~\ref{lm012912} and
conclude that the set of events $\om$, for which  
$$
\lim_{N\to+\infty}{\frak Z}^N_2\not={\frak Z}_2
$$
is contained in the set ${\cal N}$ of events $\om$ such that
\begin{equation}
\label{052812c}
\zeta\left(\underbar{T}_y(\zeta(\om))-\right)=
y<\zeta\left(\underbar{T}_y(\zeta(\om))\right),\quad\mbox{or} \quad
\zeta\left(\underbar{T}_y(\zeta(\om))\right)\le y,
\end{equation}
which again by the same arguments as used there is of null probability.

The above argument, can be continued by induction and allows us to
conclude the proof of Theorem \ref{thm013112}.
\qed

\subsection*{A  further generalization}
The argument of the present section, esentially without any modification, can be
used to prove a slight generalization of Theorem
\ref{thm013112}, that we have used in the proof of Theorem~\ref{thm011402-19}.
Suppose that $\left(
  \zeta_N(t,y)\right)_{t\ge0}$ is a sequence of processes that satisfy
$\zeta_N(0,y)=y$ and 
converge
in law, as~$N\to+\infty$,
in the $J_1$-topology over $D[0,+\infty)$ to $\left(
  \zeta(t,y)\right)_{t\ge0}$. 
We can define the consecutive crossing times ${\frak s}^N_{y,m}$ $N,m=1,2,\ldots$ for $\left(
  \zeta_N(t)\right)_{t\ge0}$ between the  half-lines $\bbR_-$ and~$\bbR_+$.
\begin{thm}
\label{thmA}
For any $y\in\bbR_*$ the random elements
\[
((\zeta_{N}(t,y))_{t\ge0},({\frak s}_{y,m}^N)_{m\ge1},(\zeta_N({\frak
  s}_{y,m}^{N},y))_{m\ge1})
\]
converge in law, as $N\to+\infty$, over
$D[0,+\infty)\times \bar\bbR^{\bbN}_+\times \bbR^{\bbN}$ with the product of the $J_1$ and standard
product topology on $(\bbR^{\bbN})^2$, to  
$((\zeta(t,y))_{t\ge0},({\frak
  u}_{y,m})_{m\ge1},(\zeta({\frak u}_{y,m},y,k)_{m\ge1}).
$ 
\end{thm}

\end{document}